\documentclass{article}

\usepackage{graphics,graphicx}
\usepackage{tikz}

\usepackage[utf8]{inputenc}
\usepackage[T1]{fontenc}

\usepackage{amsmath}
\usepackage{amssymb}
\usepackage{amsthm}

\newtheorem{thm}{Theorem}[section]
\newtheorem{cor}[thm]{Corollary}
\newtheorem{lem}[thm]{Lemma}
\newtheorem{prop}[thm]{Proposition}

\newtheorem*{thma}{Theorem A}
\newtheorem*{thmc}{Theorem C}
\newtheorem*{prop*}{Proposition B}

\theoremstyle{definition}
\newtheorem{defn}[thm]{Definition}
\newtheorem{rem}[thm]{Remark}
\newtheorem{exa}[thm]{Example}
\newtheorem{nota}[thm]{Notation}

\begin{document}

\title{Kähler-Einstein metrics on group compactifications}

\author{Thibaut Delcroix}

\date{}

\maketitle

\begin{abstract}
We obtain a necessary and sufficient condition of existence of a Kähler-Einstein metric on a 
$G\times G$-equivariant Fano compactification of a complex connected reductive group $G$
in terms of the associated polytope. This condition is not equivalent to the vanishing of the 
Futaki invariant. The proof relies on the continuity method and its translation into a real 
Monge-Ampère equation, using the invariance under the action of a maximal compact subgroup 
$K\times K$.
\end{abstract}

\section*{Introduction}

The question of the existence of a Kähler-Einstein metric on a toric Fano manifold, 
that is, a $n$-dimensional Fano manifold $X$ acted upon by $(\mathbb{C}^*)^n$ 
with an open dense orbit isomorphic to $(\mathbb{C}^*)^n$, 
was fully settled by Wang and Zhu in \cite{WZ04}. They showed that 
there exists a Kähler-Einstein metric on $X$ if and only if the Futaki invariant of $X$ 
vanishes. This invariant was known \cite{Mab87} to be linked with the barycenter of the polytope 
associated to $X$, 
so that the condition is equivalent to the fact that this barycenter is the origin.

We consider the existence of Kähler-Einstein metrics on a natural generalization of 
toric Fano manifolds: the bi-equivariant Fano compactifications of complex 
connected reductive groups. More precisely, given such a group $G$, we consider the 
Fano manifolds $X$ which admit a $G\times G$-action with an open dense orbit isomorphic
to $G$ as a $G\times G$-homogeneous space under left and right translations. We will call 
these, for simplicity, \emph{group compactifications}. 
To such a manifold is, as in the toric case, associated a polytope $P^+$, that encodes 
the information about the boundary $X\setminus G$ and the anticanonical line bundle. 

Let $G$ be a complex connected reductive group, choose $T$ a maximal torus in $G$. 
Let $M$ denote the group of characters of $T$, and let $\Phi\subset M$ be the root 
system of $(G,T)$. Choose a system of positive roots $\Phi^+$, denote by $2\rho$ 
the sum of positive roots, and by $\Xi$ the relative interior 
of the cone generated by $\Phi^+$. 
The polytope $P^+$ associated to $X$ is in the positive Weyl chamber in $M\otimes \mathbb{R}$.
Let $\mathrm{bar}_{DH}(P^+)$ denote the barycenter of $P^+$ with respect to the measure 
$\prod_{\alpha\in \Phi^+}\left<\alpha,p\right>^2dp$ where $dp$ is the Lebesgue measure 
on $M\otimes \mathbb{R}$ normalized by the lattice $M$.
Our main result is the following characterization of Kähler-Einstein Fano group 
compactifications in terms of the polytope $P^+$. 

\begin{thma}
\label{thma}
Let $X$ be a Fano compactification of $G$, then $X$ admits a Kähler-Einstein metric 
if and only if $\mathrm{bar}_{DH}(P^+)\in 2\rho+\Xi.$ 
\end{thma}

Let us remark here that $P^+$ is related to the moment polytope of $X$ from a symplectic 
point of view, and that  $\mathrm{bar}_{DH}(P^+)$ is the barycenter with respect to the 
Duistermaat-Heckman measure on this moment polytope. 

The condition obtained is computable given an explicit polytope, and we provide in this 
article both a new example of Kähler-Einstein metric and an example of group 
compactification without any Kähler-Einstein metric. 
We refer to the author's PhD thesis \cite{DelTh}, or new preprint \cite{DelKSSV},
for more examples. 
The example that does not satisfy the condition is especially interesting because its Futaki invariant 
vanishes. 
In light of the proof of the Yau-Tian-Donaldson conjecture in the Fano case 
(\cite{CDS15a,CDS15b,CDS15c} and \cite{Tia15}), 
 and the recent work of Datar and Székelyhidi \cite{DS15}, 
it means that there is a non-trivial $G\times G$-equivariant destabilizing test configuration. 

In fact, in \cite{DelKSSV} we obtain through essentially algebraic methods 
a combinatorial criterion for equivariant K-stability of a Fano spherical 
variety. Combined with \cite{DS15} it provides a combinatorial criterion 
for the existence of Kähler-Einstein metrics on spherical manifolds. 
Group compactifications are examples of spherical varieties and the 
criterion from \cite{DelKSSV} specializes to the criterion from Theorem~A 
in this case. Thus the main result of \cite{DelKSSV} provides a broad 
generalization of Theorem~A. 
On the other hand, the criterion obtained here served as the guiding 
intuition for \cite{DelKSSV}. Furthermore, the combination of the two 
articles provide an independent, and low tech, in the sense that no 
Gromov-Hausdorff convergence result is needed, proof of the Yau-Tian-Donaldson 
conjecture for Fano group compactifications. 

It should be noted that in \cite{DelKSSV} a criterion is also obtained 
for the existence of Kähler-Ricci solitons on spherical manifolds, hence
on group compactifications. The analytic methods from 
the present article could easily be extended to the case of Kähler-Ricci 
solitons to provide another proof. For simplicity, we do not provide 
details for this case. 

It was mentioned to the author by G\'abor Székelyhidi that the example 
with no Kähler-Einstein is also interesting in the study of the 
Kähler-Ricci flow on Fano manifolds. Indeed, it follows from the work of 
Chen, Sun and Wang in \cite{CSW} that the Kähler-Ricci flow on a Fano 
manifold $X$ with no Kähler-Einstein metrics should converge to some 
Kähler-Ricci soliton on the central fiber of a special test configuration 
for $X$. It is reasonable to expect, in the situation where $X$ admits an 
action of a compact group and the initial metric is invariant, that the 
corresponding test configuration is equivariant. In our example, the 
possible central fibers are known (see \cite{AK05,AB04II}), and are all 
singular. No example of such a behavior for the Kähler-Ricci flow is known, 
so it would be interesting to study this in detail. 

Let us mention now a consequence of our result in the variational setting. 
Darvas and Rubinstein proved in \cite{DR} the equivalence between the 
existence of Kähler-Einstein metrics and various notions of properness of 
functionals on spaces of Kähler potentials, allowing to take into account 
automorphisms of the manifold. Our result provides a combinatorial 
criterion for these notions of properness. More precisely, though we refer 
to \cite{DR} for details, let $\omega$ 
be a reference Kähler form in $2\pi c_1(X)$, and let $\mathcal{H}$ denote the space 
of Kähler potentials, \emph{i.e.} functions $\psi$ on $X$ such that 
$\omega + i\partial \overline{\partial} \psi$ is still a Kähler form.
Assume $X$ is Kähler-Einstein and $\mathbb{G}=\mathrm{Aut}^0(X)$ is the 
neutral component of the automorphism group of $X$. Let $\mathbb{K}$ be 
a maximal compact subgroup of $\mathbb{G}$, assume that $\omega$ is 
$\mathbb{K}$-invariant and let $\mathcal{H}^{\mathbb{K}}$ 
be the space of $\mathbb{K}$-invariant Kähler potentials. 
Then Darvas and Rubinstein prove 
that the Mabuchi functional is proper on $\mathcal{H}$ modulo the action of 
$\mathbb{G}$, and that it is proper on $\mathcal{H}^{\mathbb{K}}$ provided 
the center of $\mathbb{G}$ is finite. 
In the case of group compactifications, the automorphism group may very 
well have a center that is not finite, as it is the case for many toric 
manifolds. It is interesting to note that the methods from \cite{DR} 
should allow to prove a statement combining the two statements recalled 
above, valid with no hypotheses on the automorphism group, namely that 
the Mabuchi functional is proper on $\mathcal{H}^{\mathbb{K}}$ modulo the 
action of the normalizer $N_{\mathbb{G}}(\mathbb{K})$ of $\mathbb{K}$ 
in $\mathbb{G}$.

Let us now describe our methods. The guiding ideas are to restrict to the open dense orbit $G$, 
and consider only $K\times K$-invariant metrics. The last part is possible because of 
Matsushima's theorem \cite{Mat57} 
asserting that the isometry group of a Kähler-Einstein metric is a maximal compact subgroup 
of the connected group of automorphisms of $X$. 
We can then combine the two by using the classical $KAK$ decomposition of reductive groups.
This decomposition identifies the quotient of $G$ by $K\times K$ with the closed positive Weyl 
chamber. Another point of view is that it translates $K\times K$-invariant geometry on $G$ to 
$W$-invariant geometry on $M\otimes \mathbb{R}$, where $W$ is the Weyl group of $\Phi$.

We used this fact in \cite{DelLCT}
to associate to any positively curved hermitian metric on the anticanonical line bundle $-K_X$
a $W$-invariant convex function, whose asymptotic behavior is controlled by the polytope $P^+$.
We use this here to translate the Kähler-Einstein equation into a real Monge-Ampère equation
on a real vector space. A key point in this process is to compute the complex Hessian of a 
$K\times K$-invariant function $\varphi$ on $G$ in terms of the function it 
determines on the positive Weyl chamber. 
This computation then provides an expression of the complex Monge-Ampère 
$\mathrm{MA}_{\mathbb{C}}(\varphi)$, which is the determinant of the complex 
Hessian. 

\begin{prop*}
Let $\varphi$ be a $K\times K$-invariant smooth function $G\longrightarrow \mathbb{R}$, 
and let $u$ denote the corresponding $W$-invariant function on $M\otimes \mathbb{R}$. 
Then for any $x$ in the interior of the positive Weyl chamber, 
\[
\mathrm{MA}_{\mathbb{C}}(\varphi)(\exp x) = \mathrm{MA}_{\mathbb{R}}(u)(x)
\prod_{\alpha\in \Phi^+}\frac{\left<\alpha,\nabla u(x) \right>^2}{\sinh^2\left<\alpha,x\right>},
\]
where $\mathrm{MA}_{\mathbb{R}}(u)$ denotes the determinant of the real Hessian of $u$.
\end{prop*}

We use this result to obtain an expression of the potential, with respect to a Haar measure on $G$,
of the restriction to $G$ of a volume form on $X$ determined by a hermitian metric on $-K_X$, in 
terms of the convex potential of the hermitian metric.
This allows to translate the restriction of the Kähler-Einstein equation to $G$ as a real Monge-Ampère 
equation in a convex function $u$ reading:
\[
\mathrm{MA}_{\mathbb{R}}(u)\prod_{\alpha\in \Phi^+}\left<\alpha,\nabla u \right>^2
= e^{-u}J,
\]
where $J$ is the function defined by 
$J(x)=\prod_{\alpha\in \Phi^+}\sinh^2\left<\alpha,x\right>$.

Another consequence of our computation of the complex Hessian is an analytic derivation of a 
formula due to Brion \cite{Bri89} and Kazarnovskii \cite{Kaz87}
for the degree of an ample line bundle on a group compactification, however up to 
a multiplicative constant, which does not matter for our barycenters considerations.

Let us digress here to mention that the translation of the Kähler-Einstein equation into a real 
Monge-Ampère equation was already a key step in the toric case, but also in the case of the 
horospherical manifolds studied by Podest\`a and Spiro \cite{PS10}.
These manifolds, along with the group compactifications belong to the large family of spherical 
manifolds, which was highlighted by Donaldson in his survey \cite{Don08}
as an interesting family on which to study the existence of Kähler-Einstein metrics. 
We believe that methods similar to ours can be used for even more spherical manifolds.
On another direction, Donaldson suggests not only to study the existence of Kähler-Einstein 
metrics, but more generally of canonical metrics such as constant scalar curvature. 
We expect that our computation of the full complex Hessian can be used in relation with 
this problem. Note that Abreu obtained a formula for the scalar curvature of torus 
invariant Kähler metrics on toric manifolds \cite{Abr98} which led to important 
advances in the study of canonical metrics on toric manifolds.

The main difference between the case of toric or horospherical manifolds and the case of group 
compactifications here is the term $J$, which is independent on the function $u$. Indeed, the toric 
Monge-Ampère equation is just $\mathrm{MA}_{\mathbb{R}}(u)=e^{-u}$ and the horospherical 
case only adds well controlled gradient terms. 
We thus have to rework in depth the proof of Wang and Zhu to deal with the term $J$ and the 
walls of the Weyl chamber. This is how we obtained our condition, to replace the information given 
by the Futaki invariant in the toric case.

The common part with Wang and Zhu is the use of the classical results on the continuity method, 
and the study of a well chosen proper and strictly convex function $\nu_t$ defined in terms of the solutions
of the continuity method equation. In our case this function involves the term $J$. 
The key is then to interpret the absence of a priori estimates along the continuity method as the fact 
that the point at which the minimum of the function $\nu_t$ is attained is unbounded
as $t$ increases along the continuity method, then to interpret this 
as a condition on the polytope. 

The precise analysis of this behavior further allows to determine the maximal time of existence of 
solutions along the continuity method. This is called the greatest Ricci lower bound as Székelyhidi \cite{Sze11}
proved that it is also the supremum of all $t<1$ such that there exists a Kähler form in $c_1(X)$ with 
$\mathrm{Ric}(\omega)\geq t\omega$ where $\mathrm{Ric}(\omega)$ is the Ricci form of $\omega$.

\begin{thmc}
\label{thmc}
Assume there are no Kähler-Einstein metrics on the group compactification $X$ and let 
$R(X)$ be the greatest Ricci lower bound of $X$. Then 
\[
R(X)=\mathrm{sup}\left\{t<1 ~;~ 
\frac{1}{1-t}\left(2\rho -t\cdot \mathrm{bar}_{DH}(P^+)\right) \in -\Xi + P^+ \right\}.
\]
\end{thmc}

The particular case of toric manifolds was obtained by Li \cite{Li11}. 
Indeed, in the case of a Fano toric manifold $X$, the polytope $P^+$ is the usual 
moment polytope, the root system $\Phi$ is trivial, so that $2\rho=0$, $\Xi$ 
is restricted to the origin $\{0\}$ and the Duistermaat-Heckman barycenter 
$\mathrm{bar}_{DH}(P^+)$ is just the barycenter of $P^+$ with respect to the 
Lebesgue measure. The condition in Theorem~A thus translates as Wang and Zhu's 
condition, that is, the condition that the barycenter is the origin. 
To recover Chi Li's result, let us introduce additional notations in the 
case of groups, as illustrated in Figure~\ref{GRLB}: let $A$ be
$\mathrm{bar}_{DH}(P^+)$, let $B$ be $2\rho$ and let $C$ denote the intersection 
of the boundary of $-\Xi + P^+$ with the half line starting from $A$ in the 
direction of $B$ if it exists. 
It follows from Theorem~C that the intersection exists if and only if $R(X)<1$ and 
is then the point $B + s \overrightarrow{AB}$ with $s=R(X)/(1-R(X))$.
We may then check that $|BC|/|AC|=s/(1+s)=R(X)$.
This is precisely the expression given in \cite{Li11} in the toric case, 
up to a change in notations.

\begin{figure}
\centering
\begin{tikzpicture}[scale=1]
\draw[semithick] (1.25-0.07,2.9-0.07) -- (1.25+0.07,2.9+0.07);
\draw[semithick] (1.25+0.07,2.9-0.07) -- (1.25-0.07,2.9+0.07);
\draw (1.25+0.15,2.9-0.15) node {$A$};

\draw[thick] (-2,7/2) -- (1/2,7/2) -- (2,3) -- (4,1);
\draw[dashed] (0,7/2) -- (0,0) -- (5/2,5/2);
 
\draw[semithick] (1-0.1,3) -- (1+0.1,3);
\draw[semithick] (1,3-0.1) -- (1,3+0.1);
\draw (1-0.22,3-0.15) node {$B$};

\draw[semithick] (-2,4.2) -- (1.25,2.9);

\draw (4.2,0.7) node {$\partial (-\Xi + P^+)$};
\draw (-0.4,0.2) node {$P^+$};

\draw (-1/4,7/2+0.2) node {$C$};
\end{tikzpicture}
\caption{
\label{GRLB}
Greatest Ricci lower bound}
\end{figure}
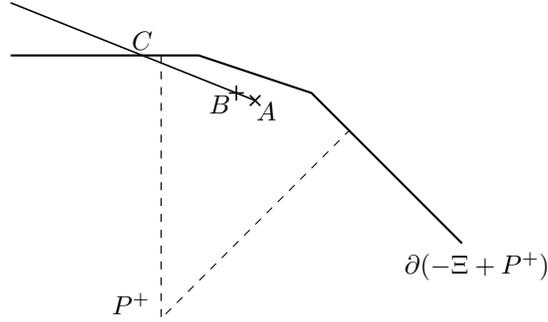

The structure of the article is as follows. 
Section~\ref{sec_CMA} is devoted to the computation of the Hessian of 
a $K\times K$-invariant function on a reductive group $G$. Proposition~B 
is obtained as a consequence in Corollary~\ref{MAC}.
Section~\ref{sec_group_comp} is a short introduction to group compactifications, 
allowing to introduce the associated polytope and relevant examples. 
The tool of convex potentials, developed in \cite{DelLCT}, is recalled here 
and applied to obtain a formula for the degree of an ample linearized line 
bundle on a group compactification. 
Section~\ref{sec_strategy} and Section~\ref{sec_m_t} provide the general strategy 
of the proofs of Theorem~A and Theorem~B, and preliminary results. 
Finally, Section~\ref{sec_obstruction} contains the proof of the necessity 
of our condition and of an upper bound on $R(X)$, and Section~\ref{sec_absence}
contains the proof that the condition is sufficient and that the upper 
bound on $R(X)$ is in fact equal to $R(X)$. 
Theorem~A is a consequence of Proposition~\ref{prop_obstruction} and 
Theorem~\ref{thm_sufficient}, and Theorem~B is a consequence of 
Proposition~\ref{prop_upper} and Section~\ref{sec_lower}.

The present article along with \cite{DelLCT} contain the main results of the author's
PhD thesis \cite{DelTh}. The author would like to thank his advisor Philippe Eyssidieux, 
and also Michel Brion for helpful discussions and the explanations on the automorphism
group of the non Kähler-Einstein example.

\section*{Notations}

Let us introduce here some notations on groups that will be used throughout 
the article. We use as reference for Lie groups the books \cite{Kna02,Hel78}, 
and \cite{Bor91} for algebraic groups.

Let $G$ be a connected complex linear reductive group.
Let $T$ be a maximal torus in $G$.
We choose a maximal compact subgroup $K$ in $G$ such that $S:=K\cap T$ is a maximal compact torus in $T$.
We denote by $\mathfrak{g}$, $\mathfrak{k}$, $\mathfrak{t}$, $\mathfrak{s}$ the respective Lie algebras
of $G$, $K$, $T$, $S$.

Let $\Phi$ denote the root system of $(G,T)$ and $W$ the corresponding Weyl group. 
We choose a system of positive roots $\Phi^+$.
We use the standard notation $\rho$ for the half sum of the positive roots.

Since $G$ is reductive, we have the \emph{Cartan decomposition} 
$\mathfrak{g}=\mathfrak{k} \oplus i\mathfrak{k}$, and 
denote by $\theta$ the corresponding Cartan involution on $\mathfrak{g}$. 
Define $\mathfrak{a} := \mathfrak{t} \cap i \mathfrak{k}$.
We identify $\mathfrak{a}$ with $N\otimes \mathbb{R}$ where $N$ is the group of one parameter 
subgroups of $T$. Denote by $\mathfrak{a}^+$ the positive \emph{open} Weyl chamber defined 
by $\Phi^+$.

Recall the \emph{$KAK$ decomposition} for reductive groups \cite[Theorem 7.39]{Kna02}: 
every element $g\in G$ can be written $g=k_1\exp(a)k_2$    
with $a\in \overline{\mathfrak{a}^+}$ uniquely determined by $g$, and $k_1,k_2 \in K$.

We can decompose $\mathfrak{a}$ in a toric part and a semisimple part: 
\[
\mathfrak{a} = \mathfrak{a}_t \oplus \mathfrak{a}_{ss},
\]
with $\mathfrak{a}_t= \mathfrak{z}(\mathfrak{g})\cap i\mathfrak{k}$ and 
$\mathfrak{a}_{ss}=\mathfrak{a}\cap [\mathfrak{g},\mathfrak{g}]$.
The Killing form of $\mathfrak{g}$, restricted to $\mathfrak{a}_{ss}$, 
gives a scalar product, and we choose a 
scalar product $\left<\cdot, \cdot\right>$ on $\mathfrak{a}$ that coincides with 
the Killing form on $\mathfrak{a}_{ss}$
and leaves $\mathfrak{a}_{ss}$ and $\mathfrak{a}_t$ orthogonal.
We use this scalar product to identify $\mathfrak{a}$ and its dual, which is also 
identified with $M\otimes \mathbb{R}$,
where $M$ is the group of characters of $T$.

\section{Complex Monge-Ampère on $G$}
\label{sec_CMA}

The aim of this section is to compute, for any $K\times K$-invariant function $\psi$ on $G$, the function 
$\mathrm{MA}_{\mathbb{C}}(\psi)$ such that 
\[
\left(\frac{i}{2\pi}\partial \overline{\partial} \psi\right)^n=\mathrm{MA}_{\mathbb{C}}(\psi)dg,
\]
where $dg$ is a Haar volume form on $G$ and $n$ is the complex dimension of $G$. 
As the notation suggests, this function is the complex Monge-Ampère
of $\psi$ in some local coordinates, \emph{i.e.} the determinant of the complex Hessian of $\psi$ 
in these coordinates. We express this Hessian and  Monge-Ampère in terms of the restriction of 
$\psi$ to $\exp(\mathfrak{a})$, using the $K\times K$-invariance and the $KAK$-decomposition.

\subsection{Complex Hessian matrix on $G$}
\label{CHess}

We first choose a special basis of $\mathfrak{g}\simeq T_eG$.
Recall the root decomposition:
\[
\mathfrak{g}=\mathfrak{t}\oplus \bigoplus_{\alpha \in \Phi} \mathfrak{g}_{\alpha}
\]
where, if $\mathrm{ad}$ denotes the adjoint representation of $\mathfrak{g}$,
\[
\mathfrak{g}_{\alpha}:= \left\{x\in \mathfrak{g} ~;~ 
\mathrm{ad}(h)(x)=\left<\alpha,h\right>x \quad \forall h \in \mathfrak{t}\right\}.
\]
We can choose a set $\{e_{\alpha}\}_{\alpha\in \Phi}$ of elements of $\mathfrak{g}$ such 
that $e_{-\alpha}=-\theta(e_{\alpha})$, $[e_{\alpha}, e_{-\alpha}]=\alpha$ and $e_{\alpha}$
generates the complex line $\mathfrak{g}_{\alpha}$ \cite[Chapter VI, Lemma 3.1]{Hel78}.

We would like to choose a complex basis of $\mathfrak{g}$ which is also a real basis of $\mathfrak{k}$.
Define, for $\alpha\in \Phi^+$,
\[
\mathfrak{k}_{\alpha}:= \left\{x\in \mathfrak{k} ~;~ 
\mathrm{ad}(h)^2(x)=\left<\alpha,h\right>^2x \quad \forall h \in \mathfrak{s}\right\}.
\]
Then each $\mathfrak{k}_{\alpha}$ is of real dimension two and we have 
$\mathfrak{g}_{\alpha}\oplus \mathfrak{g}_{-\alpha} =
\mathfrak{k}_{\alpha}\oplus i \mathfrak{k}_{\alpha}$.
We obtain a real basis of $\mathfrak{k}_{\alpha}$ by considering 
$e_{\alpha}+\theta(e_{\alpha})=e_{\alpha}-e_{-\alpha}$ and 
$ie_{\alpha}+\theta(ie_{\alpha})=ie_{\alpha}+ie_{-\alpha}$.
We complete this with a real basis of $\mathfrak{s}$ to obtain a real basis of 
$\mathfrak{k}= \mathfrak{s}\oplus \bigoplus_{\alpha\in \Phi^+}\mathfrak{k}_{\alpha}$, 
and thus a complex basis of $\mathfrak{g}$. 
We denote this complex basis of $\mathfrak{g}$ by $(l_j)_{j=1}^n$.

Since the exponential map is a biholomorphism from a neighborhood of 
$0 \in \mathfrak{g}$ to a neighborhood of the neutral element $e$ in $G$,
we get holomorphic coordinates on $G$ near $e$.
Then, using a translation on the right by $g \in G$, 
this defines holomorphic coordinates on a neighborhood of $g$.
More precisely, 
the map corresponding to the local coordinates is the map 
$\mathbb{C}^n \rightarrow G$ defined by 
\[
(z_1,\ldots, z_n) \mapsto \exp(z_1l_1+\cdots z_nl_n)g.
\]

We will compute the complex Hessian with respect to these coordinates.
If $\psi$ is a function on $G$ we denote
by $\mathrm{Hess}_{\mathbb{C}}(\psi)(g)$ the complex Hessian and
 by $\mathrm{MA}_{\mathbb{C}}(\psi)(g)$
the determinant of the complex Hessian of $\psi$ at $g$, called the complex Monge-Ampère, 
in both cases with respect to the coordinates given above. 

If $f$ is a function on $\mathfrak{a}$, then we denote by 
$\mathrm{MA}_{\mathbb{R}}(f)(x)$ the determinant of its real Hessian at $x$, and
by $\nabla f$ the gradient of $f$ with respect to the scalar product $\left<\cdot,\cdot \right>$ 
on $\mathfrak{a}$.

By the $KAK$ decomposition, a $K\times K$-invariant function $\psi$ on $G$ is completely 
determined by the function $f$ defined on $\mathfrak{a}$ by $f(x)=\psi(\exp x)$.

Let us begin by a remark that shows that even though we use a specific set of coordinates to compute the 
complex Hessian, the complex Monge-Ampère thus obtained is the potential 
of $\frac{i}{2\pi}\partial \overline{\partial} \psi$ with respect to a Haar 
measure, up to a multiplicative constant.
\begin{rem}
\label{ChoiceofHaar}
Let $dg$ be a Haar volume form on $G$ (thus a $G\times G$-invariant volume form, 
because a reductive group is unimodular \cite[Section VIII.2]{Kna02}). 
Then there exists a constant $C>0$ such that for any $K\times K$-invariant 
function $\psi$ on $G$, we have, 
\[
\left(\frac{i}{2\pi}\partial \overline{\partial} \psi\right)^n =
C \mathrm{MA}_{\mathbb{C}}(\psi) dg.
\]
Indeed, in local coordinates $z_j$, we can write
\[
\frac{i}{2\pi}\partial \overline{\partial} \psi =
\sum_{j,k} \frac{i}{2\pi} \frac{\partial^2 \psi}{\partial z_j \partial \overline{z_k}}dz_j\wedge d\overline{z_k},
\]
so in the coordinates above, we have locally 
\[
\left(\frac{i}{2\pi}\partial \overline{\partial} \psi\right)^n =
\frac{1}{(2\pi)^n} \mathrm{MA}_{\mathbb{C}}(\psi) 
(i^n dz_1\wedge d\overline{z_1}\wedge\ldots \wedge dz_n\wedge d\overline{z_n}).
\]
But by the construction of the coordinates, the global form defined locally as 
$i^n dz_1\wedge d\overline{z_1}\wedge\ldots \wedge dz_n\wedge d\overline{z_n}$
is a Haar volume form.
To conclude, remark that all Haar volume forms are positive scalar multiples of one another.
\end{rem}

\begin{thm}
Let $\psi$ be a $K\times K$ invariant function on $G$, and $f$ the associated function 
on $\mathfrak{a}$. Then in the coordinates above and for $a\in \mathfrak{a}^+$, 
the complex Hessian matrix of $\psi$ is diagonal by blocks, equal to:
\[
\mathrm{Hess}_{\mathbb{C}}(\psi)(\exp(a)) = 
\begin{pmatrix}
\frac{1}{4}\mathrm{Hess}_{\mathbb{R}}(f)(a)& 0 &  & & 0 \\
 0 & M_{\alpha_1}(a) & & & 0 \\
 0 & 0 & \ddots & & \vdots \\
\vdots & \vdots & & \ddots & 0\\
 0 & 0 &  & & M_{\alpha_p}(a)\\ 
\end{pmatrix}
\]
where the $(\alpha_i)_{i\in\{1,\ldots,p\}}$ describe the positive roots of $\Phi$ and 
$M_{\alpha}$ is defined by:
\[
M_{\alpha}(a) = \frac{1}{2}\left<\alpha , \nabla f(a)\right>
\begin{pmatrix}
\coth\left<\alpha , a\right> & i \\
-i & \coth\left<\alpha ,a\right> \\
\end{pmatrix}.
\]
\end{thm}

As a corollary, we obtain Proposition~B:

\begin{cor}
\label{MAC}
Let $\psi$ be a $K\times K$ invariant function on $G$, and $f$ the associated function 
on $\mathfrak{a}$. Then in the coordinates above and at $a\in \mathfrak{a}^+$,
if $r$ denotes the rank of $G$ and $p$ the number of positive roots, we have
\[
\mathrm{MA}_{\mathbb{C}}(\psi)(\exp(a))= \frac{1}{4^{r+p}}
\mathrm{MA}_{\mathbb{R}}(f)(a)\frac{1}{J(a)}\prod_{\alpha \in \Phi^+} \left<\alpha,\nabla f(a)\right>^2 
\]
where we denote by $J$ the function defined on $\mathfrak{a}$ by 
\[
J(a):= \prod_{\alpha \in \Phi^+} \sinh^2\left<\alpha ,a\right>.
\]
\end{cor}

\begin{proof}
Since $\mathrm{MA}_{\mathbb{R}}(f)(a)=\mathrm{det}(\mathrm{Hess}_{\mathbb{R}}(f)(a))$, 
we just have to compute the determinant of $M_{\alpha}$.
This is 
\begin{align*}
\mathrm{det}(M_{\alpha}) 
& =\left(\frac{1}{2}\left<\alpha ,\nabla f(a)\right>\right)^2(\coth^2\left<\alpha ,a\right>-1)\\
& = \frac{1}{4}\left<\alpha ,\nabla f(a)\right>^2
	\frac{\cosh^2\left<\alpha ,a\right>-\sinh^2\left<\alpha ,a\right>}{\sinh^2 \left<\alpha ,a\right>}\\
& = \frac{1}{4}\left<\alpha ,\nabla f(a)\right>^2\frac{1}{\sinh^2\left<\alpha ,a\right>}
\qedhere
\end{align*}
\end{proof}

\begin{exa}
Consider the case $G=\mathrm{PSL}_2(\mathbb{C})$. Then $\mathfrak{a}^+ \simeq \mathbb{R}^*_+$, 
and there is only one positive root that we can identify with the identity on $\mathbb{R}$.
Then 
\[
\mathrm{Hess}_{\mathbb{C}}(\psi)(\exp(a)) = \frac{1}{2}
\begin{pmatrix}
f''(a)/2 & 0& 0\\
0& f'(a)\coth(a) & if'(a) \\
0& -if'(a) & f'(a)\coth(a)\\
\end{pmatrix}
\]
and the complex Monge-Ampère reads: 
\[
\mathrm{MA}_{\mathbb{C}}(\psi)(\exp(a)) = \frac{1}{16}f''(a)(f'(a))^2\frac{1}{\sinh^2(a)}.
\]
\end{exa}

The rest of the section is devoted to the proof of the theorem.
The technique of the proof is based on the work of Bielawski \cite{Bie04}. 
In particular, the idea to use the decomposition in Lemma~\ref{decomposition} and the 
Baker-Campbell-Hausdorff formula appears in this article.
We begin by introducing these two tools.

\subsection{The Baker-Campbell-Hausdorff formula}

As a formal series in the variables  $x$ and $y$, the logarithm of 
$\exp(x)\exp(y)$ is well defined. We denote this by $\mathrm{BCH}(x,y)$.
The Baker-Campbell-Hausdorff formula is the following.

\begin{prop}{\cite[Theorem X.3.1]{Hoc65}}
\label{BCH}
There exists a neighborhood $U$ of $0$ in $\mathfrak{g}$ such that for 
all $x$ and $y$ in $U$, $\mathrm{BCH}(x,y)$ is convergent and defines 
an element of $\mathfrak{g}$, and we have 
\[
\exp(x)\exp(y)=\exp(\mathrm{BCH}(x,y)).
\]
\end{prop}

Furthermore we know explicitly the terms of $\mathrm{BCH}(x,y)$. We will only use
the following:
\[
\mathrm{BCH}(x,y)=x+y+\frac{1}{2}[x,y]+O
\]
where $O$ denotes terms of order higher than 2 in $x$ and $y$.

\subsection{A decomposition in $\mathfrak{g}$}

Let $a\in \mathfrak{a}^+$. Let $\mathrm{Exp}(\mathrm{ad}(a))$ be the linear application 
$\mathfrak{g} \rightarrow \mathfrak{g}$ defined by 
\[
\mathrm{Exp}(\mathrm{ad}(a))(x)=\sum_{n=0}^{\infty}\frac{\mathrm{ad}(a)^n(x)}{n!}.
\]
Recall that $G$ acts on $\mathfrak{g}$ through the adjoint action $\mathrm{Ad}$, and 
that we have the general relation 
\[
\mathrm{Exp}(\mathrm{ad}(a))(x)=\mathrm{Ad}(\exp a)(x)
\]
for $x\in \mathfrak{g}$.

\begin{lem}
\label{decomposition}
Let $l\in  \mathfrak{g}$ and $a\in \mathfrak{a}^+$. Then 
\begin{enumerate}
\item there exists $A \in \mathfrak{k}$, $B\in \mathfrak{a}$ 
and $C\in \mathrm{Ad}(\exp a)\left(\bigoplus_{\alpha\in \Phi^+}\mathfrak{k}_{\alpha}\right)$ such that 
\[
l=A+B+C;
\]
\item if $l\in \bigoplus_{\alpha \in \Phi} \mathfrak{g}_{\alpha}$ then $B=0$;
\item if $l\in \mathfrak{k}_{\alpha}$, and $l'$ denotes $\mathrm{ad}(a)(il)/\left<\alpha ,a\right>$, 
then $l'\in \mathfrak{k}_{\alpha}$ and the decomposition above for the element 
$il \in \mathfrak{g}_{\alpha} \oplus \mathfrak{g}_{-\alpha}$ reads
\[
il= -(\cosh\left<\alpha , a\right>) l'+\frac{1}{\sinh(\left<\alpha , a\right>)}\mathrm{Ad}(\exp a)(l');
\]
\item if $l=e_{\alpha}+\theta(e_{\alpha})$ then $l'=ie_{\alpha}-i\theta(e_{\alpha})$;
\item if $l=ie_{\alpha}-i\theta(e_{\alpha})$ then $l'=-e_{\alpha}-\theta(e_{\alpha})$.
\end{enumerate}
\end{lem}
In the statement, the result is more and more precise as we know more precisely the element 
considered. In the proof we will begin by the very precise case of the elements of the basis and work 
our way up by linearity.

\begin{proof}
Let $a\in \mathfrak{a}^+$.
We begin by the two last points.
By definition of $l'$, we have, if $l=e_{\alpha}+\theta(e_{\alpha})$, 
\begin{align*}
l'      & =\mathrm{ad}(a)(il)/\left<\alpha , a\right>\\
 	& = \mathrm{ad}(a)(ie_{\alpha})/\left<\alpha , a\right>
		+ \mathrm{ad}(a)(i\theta(e_{\alpha}))/\left<\alpha , a\right>\\
	& = ie_{\alpha}-i\theta(e_{\alpha})\\
\intertext{and if $l=ie_{\alpha}-i\theta(e_{\alpha})$,}
l' & =\mathrm{ad}(a)(il)/\left<\alpha , a\right> \\
	& = \mathrm{ad}(a)(-e_{\alpha})/\left<\alpha , a\right>
		+\mathrm{ad}(a)(\theta(e_{\alpha}))/\left<\alpha , a\right>\\
	& = -e_{\alpha}-\theta(e_{\alpha}).
\end{align*}
In particular, in both cases, $l'$ is in $\mathfrak{k}_{\alpha}$.
By linearity this is also true of $l'$ for any $l\in \mathfrak{k}_{\alpha}$.

To prove the decomposition in the third point, it suffices to compute that, 
using the definition of $\mathfrak{k}_{\alpha}$,
\begin{align*}
\mathrm{Exp}(\mathrm{ad}(a))(l') 
& =(\cosh \left<\alpha , a\right>)l' + (\sinh \left<\alpha , a\right>)il \\
&=\mathrm{Ad}(\exp a)(x).
\end{align*}

Then by linearity the first point holds true for any 
$l \in i \bigoplus_{\alpha \in  \Phi^+} \mathfrak{k}_{\alpha}$, with $B=0$.
But we have
\[ 
\bigoplus_{\alpha \in  \Phi} \mathfrak{g}_{\alpha} =
\bigoplus_{\alpha \in  \Phi^+} \mathfrak{k}_{\alpha} 
\oplus i \bigoplus_{\alpha \in  \Phi^+} \mathfrak{k}_{\alpha}, 
\]
so we have the decomposition for any 
$l\in \bigoplus_{\alpha \in  \Phi} \mathfrak{g}_{\alpha}$,
with $B=0$.

Finally for $l \in \mathfrak{t}$, it suffices to decompose $l$ along 
$\mathfrak{t}=\mathfrak{s}\oplus \mathfrak{a}$.
By linearity and the root decomposition, we obtain the proposition for any 
$l\in \mathfrak{g}$.
\end{proof}

\subsection{Using the Baker-Campbell-Hausdorff formula}

We want to compute the complex Hessian of $\psi$ in the chosen system of coordinates, at a point 
$\exp (a)$ for $a$ in the open Weyl chamber $\mathfrak{a}^+$.
If $l_1$ and $l_2$ are two vectors in the chosen basis of $\mathfrak{k}$, we thus want to compute:
\[
H_{l_1,l_2}(a):= \left. \frac{\partial^2}{\partial z_1 \partial \overline{z_2}} \right|_{z_1,z_2=0}
\psi(\exp(z_1l_1 + z_2l_2) \exp(a)).
\]

There are different cases, according to the subspaces where $l_1$ and $l_2$ lie.
We will first describe the part of the argument that is used in all cases, which relies on 
the Baker-Campbell-Hausdorff formula, and then deal with each case separately.

Using the decomposition from Lemma~\ref{decomposition} on $z_1l_1+z_2l_2$ 
we can write 
\[
z_1l_1+z_2l_2 = A_1+B_1+C_1
\]
with $A_1\in \mathfrak{k}$, $B_1\in \mathfrak{a}$ 
and $C_1\in \mathrm{Ad}(\exp a)(\mathfrak{k})$, 
and all are of homogeneous degree one in $z_1$ and $z_2$.
Let 
\[
D_1 = \frac{1}{2}([B_1,A_1]+[C_1,A_1]+[C_1,B_1]),
\] 
it is of order two in $z_1$ and $z_2$.

Let us now use again Lemma~\ref{decomposition} to get 
\[
D_1 = A_2+B_2+C_2.
\]
with $A_2\in \mathfrak{k}$, $B_2\in \mathfrak{a}$ 
and $C_2\in \mathrm{Ad}(\exp a)(\mathfrak{k})$, 
and all are of homogeneous degree two in $z_1$ and $z_2$.

The Baker-Campbell-Hausdorff formula allows to prove the following lemma,
which can be seen as an explicit infinitesimal $KAK$ decomposition.
Note that to lighten the notations we do not write the dependence on $z_1$, $z_2$,
but all the terms defined above $A_j$, $B_j$, $C_j$ and $D_1$ are in fact functions
of these two complex variables. 

\begin{lem}
\label{UseBCH}
We can write 
\[
\exp(z_1l_1 + z_2l_2) \exp(a) = k_1 \exp(B_1+B_2+a+O) k_2
\]
where $k_1, k_2 \in K$, and $O$ denotes terms of order greater than two in $z_1$ and $z_2$.
\end{lem}

\begin{proof}
We begin by applying Proposition~\ref{BCH} to 
$\exp(-A_1)\exp(A_1+B_1+C_1)$, and get that this is equal to 
\[
\exp\left(B_1+C_1+[-A_1,B_1+C_1]/2+O_1\right),
\]
where $O_1$ denotes terms of order greater than 2 in $z_1$ and $z_2$.

Then we multiply on the right by $\exp(-C_1)$ and get, with Proposition~\ref{BCH}
again, 
\[
\exp\left(B_1+[-A_1,B_1+C_1]/2+[B_1,-C_1]/2+O_2\right),
\]
where $O_2$ denotes terms of order greater than 2 in $z_1$ and $z_2$.
By definition of $D_1$, we have proved 
\[
\exp(-A_1)\exp(z_1l_1 + z_2l_2)\exp(-C_1)= \exp\left( B_1+D_1+O_2\right).
\]

Recall that $D_1=A_2+B_2+C_2$, and that all of these are of degree two in $z_1$ and $z_2$.
We apply another time the Proposition~\ref{BCH}, to 
$\exp(-A_2)\exp(B_1+D_1+O_2)$, but here we only need to use the first term 
in the development of $\mathrm{BCH}$. We might say that $A_2$ commutes 
up to order two with elements of degree greater or equal to one in $z_1$, $z_2$.
We get 
\[
\exp(-A_2)\exp(B_1+D_1+O_2)=\exp(B_1+B_2+C_2+O_3),
\]
where $O_3$ denotes terms of order greater than 2 in $z_1$ and $z_2$.

One further use of the Baker-Campbell-Hausdorff formula yields
\[
\exp(-A_2)\exp(B_1+D_1+O_2)\exp(-C_2)=\exp(B_1+B_2+O_4),
\]
where $O_4$ denotes terms of order greater than 2 in $z_1$ and $z_2$.

Consider now $\exp(C_2)\exp(C_1)$.
Since  $C_1, C_2 \in \mathrm{Ad}(\exp a)(\mathfrak{k})$, we have 
\[
\exp(C_2)\exp(C_1)=\exp(a)k_2\exp(-a)
\]
for some $k_2\in K$.
On the other hand, we have $k_1:=\exp(A_1)\exp(A_2)\in K$.

Summing up we have proved that 
\[
\exp(z_1l_1 + z_2l_2)=k_1\exp(B_1+B_2+O_4)\exp(a)k_2\exp(-a).
\]
But then 
\[
\exp(z_1l_1 + z_2l_2)\exp(a)=k_1\exp(B_1+B_2+O_4)\exp(a)k_2,
\]
and one last application of Proposition~\ref{BCH} gives the lemma, because 
$B_1, B_2$ and $a$ commute:
\[
\exp(z_1l_1 + z_2l_2) \exp(a) = k_1 \exp(B_1+B_2+a+O) k_2
\]
where $O$ denotes terms of order greater than 2 in $z_1$ and $z_2$.
\end{proof}

\begin{lem}
\label{psiaf}
We have 
\[
H_{l_1,l_2}(a)=
\left. \frac{\partial^2}{\partial z_1 \partial \overline{z_2}} \right|_{0}
f(a+B_2+B_1).
\]
\end{lem}

\begin{proof}
We first use $K\times K$-invariance of $\psi$ and Lemma~\ref{UseBCH} to write 
\[
\psi(\exp(z_1l_1+z_2l_2)\exp(a))=\psi(\exp(a+B_1+B_2+O)).
\]

Then 
\begin{align*}
H_{l_1,l_2}(a)&=\left. \frac{\partial^2}{\partial z_1 \partial \overline{z_2}} \right|_{z_1,z_2=0}
\psi(\exp(z_1l_1 + z_2l_2) \exp(a))\\
& = \left. \frac{\partial^2}{\partial z_1 \partial \overline{z_2}} \right|_{0}
\psi(\exp(a+B_1+B_2+O))\\
\intertext{because $O$ is of order greater than two, this becomes}
& = \left. \frac{\partial^2}{\partial z_1 \partial \overline{z_2}} \right|_{0}
\psi(\exp(a+B_1+B_2))\\
\intertext{since $a+B_1+B_2 \in \mathfrak{a}$, this is}
& = \left. \frac{\partial^2}{\partial z_1 \partial \overline{z_2}} \right|_{0}
f(a+B_1+B_2)
\qedhere
\end{align*}
\end{proof}

It remains to determine $B_1+B_2$ for all coefficients of the Hessian, and then to compute 
the coefficient. 
For that, since we want to reduce to real coordinates, we recall that 
if $z_1=x_1+iy_1$ and $z_2=x_2+iy_2$ then 
\[ 
\frac{\partial^2}{\partial z_1 \partial \overline{z_2}} =
\frac{1}{4}\left( \frac{\partial^2}{\partial x_1 \partial x_2} + 
\frac{\partial^2}{\partial y_1 \partial y_2}\right) + 
\frac{i}{4}\left(\frac{\partial^2}{\partial x_1 \partial y_2} - 
\frac{\partial^2}{\partial y_1 \partial x_2} \right).
\]

\subsection{Determining $H_{l_1,l_2}(a)$}

\begin{lem}
Assume $l_1,l_2 \in \mathfrak{s}$. Then 
$H_{l_1,l_2}(a)$ is the corresponding coefficient of $\mathrm{Hess}_{\mathbb{R}}(f)(a)/4$: 
\[
H_{l_1,l_2}(a) = 
\frac{1}{4}\left. \frac{\partial^2}{\partial y_1 \partial y_2} \right|_{0}
f(a+y_1il_1+y_2il_2).
\]
\end{lem}

\begin{proof}
In this case we have 
$z_1l_1+z_2l_2 = A_1+B_1+0$
with $A_1=x_1l_1+x_2+l_2\in \mathfrak{s}$ and $B_1= y_1l_1+y_2l_2\in \mathfrak{a}$, 
and $A_1$ and $B_1$ commute, 
so $D_1=0$ and $B_2=0$.

Then by Lemma~\ref{psiaf},
\begin{align*}
H_{l_1,l_2}(a)& =
\left. \frac{\partial^2}{\partial z_1 \partial \overline{z_2}} \right|_{0}
f(a+y_1il_1+y_2il_2) \\
& = \frac{1}{4}\left. \frac{\partial^2}{\partial y_1 \partial y_2} \right|_{0}
f(a+y_1il_1+y_2il_2)
\qedhere 
\end{align*}
\end{proof}

\begin{lem}
Assume $l_1 \in \mathfrak{k}_{\alpha}$ and $l_2 \in \mathfrak{s}$, 
then $H_{l_1,l_2}(a)=0$.
\end{lem}

\begin{proof}

Let us first determine the $A_1, B_1, C_1$ such that $z_1l_1+z_2l_2=A_1+B_1+C_1$. 
Using Lemma~\ref{decomposition}, write 
\[
il_1=-(\coth\left<\alpha , a\right>)l_1' + 
\frac{1}{\sinh\left<\alpha , a\right>}(\mathrm{Ad}(\exp a)(l_1')
\]
with $l_1'=\mathrm{ad}(a)(il)/\left<\alpha , a\right>$.

Then we have 
\begin{align*}
A_1 & =x_1l_1 -y_1(\coth\left<\alpha , a\right>)l_1' +x_2l_2  \\
B_1 & =y_2 il_2 \\
C_1 & = \frac{y_1}{\sinh\left<\alpha , a\right>}
	(\mathrm{Ad}(\exp a)(l_1')=y_1il_1+y_1(\cosh\left<\alpha , a\right>)l_1'
\end{align*}

We must now compute $D_1=\left( [B_1,A_1]+[C_1,A_1]+[C_1,B_1] \right)/2$.
In fact we must only determine $B_2$ which is the part of $D_1$ that lies in $\mathfrak{a}$.

We have 
\begin{align*}
[B_1,A_1] & = [y_2 il_2, x_1l_1 -y_1(\coth\left<\alpha , a\right>)l_1' +x_2l_2] \\
& = -y_1y_2(\coth \left<\alpha , a\right>) [il_2, l_1']+x_1y_2[il_2,l_1] 
\end{align*}
Now $il_2 \in \mathfrak{a}$ and 
$l_1,l_1' \in \mathfrak{k}_{\alpha}\subset \mathfrak{g}_{\alpha} \oplus \mathfrak{g}_{-\alpha}$
so $ [il_2, l_1'],[il_2,l_1] \in \mathfrak{g}_{\alpha} \oplus \mathfrak{g}_{-\alpha}$, and 
the third point of Lemma~\ref{decomposition} applies to show that the $\mathfrak{a}$ component 
of $[B_1,A_1]$ is zero.

For the second part, write
\begin{align*}
[C_1,A_1] & = x_1y_1(\cosh\left<\alpha , a\right>) [l_1',l_1] 
-y_1^2(\coth\left<\alpha , a\right>)[il_1,l_1'] \\
& \qquad + x_2y_1[il_1,l_2] 
+ x_2y_1(\cosh\left<\alpha , a\right>)[l_1',l_2]
\end{align*}
We have here $[l_1',l_1], [l_1',l_2] \in \mathfrak{k}$, 
and $[il_1,l_2]\in  \mathfrak{g}_{\alpha} \oplus \mathfrak{g}_{-\alpha}$ as above, 
so only $[il_1,l_1']$ matters.
By the properties of the root decomposition, 
\[
[il_1,l_1']\in (\mathfrak{g}_{-2\alpha} \oplus 
\mathfrak{g}_0 \oplus \mathfrak{g}_{2\alpha})\cap i\mathfrak{k}
\]
and $\mathfrak{g}_{-2\alpha}=\mathfrak{g}_{2\alpha}=\{0\}$. So $[il_1,l_1']\in \mathfrak{a}$.
But in fact we don't need to determine it more explicitly because it appears as a term in 
$y_1^2$ and these are ignored in the computation of $\partial \overline{\partial}$.

For the third part, 
\[
[C_1,B_1]= y_1y_2[il_1,il_2]+y_1y_2 (\cosh\left<\alpha , a\right>)[l_1',il_2] 
\]
with $[il_1,il_2] \in \mathfrak{k}$ and $[l_1',il_2]\in \mathfrak{g}_{\alpha} \oplus \mathfrak{g}_{-\alpha}$
so there is no contribution to $B_2$.

We have thus proved that 
\[
B_2=-\frac{1}{2}y_1^2(\coth \left<\alpha , a\right>)[il_1,l_1'].
\]

Lemma~\ref{psiaf} now gives 
\begin{align*}
H_{l_1,l_2}(a) & =
\left. \frac{\partial^2}{\partial z_1 \partial \overline{z_2}} \right|_{0}
f(a+B_2+B_1)  \\
& = 
\left. \frac{\partial^2}{\partial z_1 \partial \overline{z_2}} \right|_{0}
f\left(a+y_2il_2-y_1^2(\coth \left<\alpha , a\right>)[il_1,l_1']/2\right) \\
&=0
\qedhere
\end{align*}
\end{proof}

\begin{lem}
Assume $l_1\in \mathfrak{k}_{\alpha_1}$ and $l_2\in \mathfrak{k}_{\alpha_2}$, with 
$\alpha_1 \neq \alpha_2$ positive roots.
Then $H_{l_1,l_2}(a)=0$.
\end{lem}

\begin{proof}
Using Lemma~\ref{decomposition}, we write 
\begin{align*}
il_1 & =-(\coth\left< \alpha_1 , a\right>)l_1' + 
\frac{1}{\sinh\left< \alpha_1 , a\right>}\mathrm{Ad}(\exp a)(l_1')\\
il_2 & =-(\coth\left< \alpha_2 , a\right>)l_2' + 
\frac{1}{\sinh\left< \alpha_2 , a\right>}\mathrm{Ad}(\exp a)(l_2') 
\end{align*}
with $l_1'=\mathrm{ad}(a)(il_1)/\left< \alpha_1 , a\right>$ and 
$l_2'=\mathrm{ad}(a)(il_2)/\left< \alpha_2 , a\right>$.

Then we have 
\begin{align*}
A_1 & 
  = x_1l_1+x_2l_2 -y_1(\coth\left< \alpha_1 , a\right>)l_1'-y_2(\coth\left< \alpha_2 , a\right>)l_2'   \\
B_1 & 
  = 0\\
C_1 & 
  = y_1\frac{1}{\sinh\left< \alpha_1 , a\right>}
	(\mathrm{Ad}(\exp a)(l_1')+y_2\frac{1}{\sinh\left< \alpha_2 , a\right>}(\mathrm{Ad}(\exp a)(l_2')\\
D_1 & =  [C_1,A_1]/2\\
 & =  [y_1il_1+y_2il_2+y_1(\cosh\left< \alpha_1 , a\right>)l_1'+y_2(\cosh\left< \alpha_2 , a\right>)l_2' , A_1]/2
\end{align*}
We have $y_1(\cosh \left< \alpha_1 , a\right>)l_1'+y_2(\cosh\left< \alpha_2 , a\right>)l_2'$ and 
$A_1$ in $\mathfrak{k}$, so their bracket remains in $\mathfrak{k}$ and doesn't appear 
in $B_2$.
We compute $[y_1il_1+y_2il_2, A_1]$ which is equal to 
\begin{align*}
  &  
x_2y_1[il_1,l_2]
	-y_1^2(\coth\left< \alpha_1 , a\right>)[il_1,l_1']
	-y_1y_2(\coth\left< \alpha_1 , a\right>)[il_2,l_1']\\
& \qquad +x_1y_2[il_2,l_1] -y_2^2(\coth\left< \alpha_2 , a\right>)[il_2,l_2'] 
		-y_1y_2(\coth\left< \alpha_2 , a\right>)[il_1,l_2'].
\end{align*}

Again the properties of the root decomposition tell us that 
$[il_1,l_2]$, $[il_1,l_2']$, $[il_2,l_1],$ and $[il_2,l_1']$ are in 
$\bigoplus_{\alpha \in \Phi} \mathfrak{g}_{\alpha}$, 
so the corresponding terms do not contribute to $B_2$.
As before, $[il_1,l_1']$ and $[il_2,l_2']$ are in $\mathfrak{a}$, so we get 
\[
B_2=\frac{1}{2}(-y_1^2(\coth\left< \alpha_1 , a\right>)[il_1,l_1']-
y_2^2(\coth\left< \alpha_2 , a\right>)[il_2,l_2']).
\]

Applying Lemma~\ref{psiaf}, we get 
\begin{align*}
H_{l_1,l_2}(a) & = 
\left. \frac{\partial^2}{\partial z_1 \partial \overline{z_2}} \right|_{0}
f(a+B_2+B_1)  \\
& = 
\left. \frac{\partial^2}{\partial z_1 \partial \overline{z_2}} \right|_{0}
f\left(a-y_1^2\coth(\left< \alpha_1 , a\right>)[il_1,l_1']/2-y_2^2\coth(\left< \alpha_2 , a\right>)[il_2,l_2']/2\right) \\
&=0
\qedhere
\end{align*}
\end{proof}

Suppose now that $\alpha_1=\alpha_2=\alpha$. 
The subspace $\mathfrak{k}_{\alpha}$ is two dimensional, 
and we have chosen
a basis formed by the vectors
$e_{\alpha}+\theta(e_{\alpha})$ and $ie_{\alpha}-i\theta(e_{\alpha})$.

First we deal with the case when $l_1\neq l_2$.

\begin{lem}
Suppose $l_1=e_{\alpha}+\theta(e_{\alpha})$ and $l_2=ie_{\alpha}-i\theta(e_{\alpha})$. Then
\[
H_{l_1,l_2}(a)= \frac{i}{2}\left<\alpha , \nabla f(a)\right> = -H_{l_2,l_1}(a).
\]
\end{lem}

\begin{proof}
Using Lemma~\ref{decomposition}, we write, just as in the previous proof
\begin{align*}
il_1 & =-(\coth\left<\alpha , a\right>)l_1' + 
		\frac{1}{\sinh\left<\alpha , a\right>}\mathrm{Ad}(\exp a)(l_1')\\
il_2 & =-(\coth\left<\alpha , a\right>)l_2' + 
		\frac{1}{\sinh\left<\alpha , a\right>}\mathrm{Ad}(\exp a)(l_2')
\end{align*}
with $l_1' =l_2$  and  $l_2'= -l_1$.

Then we have 
\begin{align*}
A_1 & = (x_1+y_2\coth\left<\alpha , a\right>)l_1+(x_2 -y_1\coth\left<\alpha , a\right>)l_2   \\
B_1 & = 0\\
C_1 & = y_1\frac{1}{\sinh\left<\alpha , a\right>}\mathrm{Ad}(\exp a)(l_1')+y_2\frac{1}{\sinh\left<\alpha , a\right>}\mathrm{Ad}(\exp a)(l_2')\\
D_1 & = [C_1,A_1]/2\\
 & =  [y_1il_1+y_2il_2
		+y_1\cosh(\left<\alpha , a\right>)l_2
			-y_2\cosh(\left<\alpha , a\right>)l_1 , A_1]/2
\end{align*}

Once again the bracket of 
$y_1(\cosh\left<\alpha , a\right>)l_2-y_2(\cosh\left<\alpha , a\right>)l_1$ with $A_1$ yields 
only terms in $\mathfrak{k}$ so we compute $[y_1il_1+y_2il_2,A_1]$, which is equal to 
\[
y_2(x_1+y_2\coth\left<\alpha , a\right>)[il_2,l_1]+y_1(x_2-y_1\coth\left<\alpha , a\right>)[il_1,l_2].
\]
Using the explicit choices of $l_1$ and $l_2$ we have 
\begin{align*}
-[il_2,l_1]=[il_1,l_2] & = [i(e_{\alpha}+\theta(e_{\alpha})),ie_{\alpha}-i\theta(e_{\alpha})]\\
	& = 2[e_{\alpha},\theta(e_{\alpha})]\\
	& = -2[e_{\alpha},e_{-\alpha}]\\
	& = -2\alpha.
\end{align*}

Finally we have 
\[
B_2=\left(y_2x_1+y_2^2\coth\left<\alpha , a\right>-y_1x_2+y_1^2\coth\left<\alpha , a\right>\right)\alpha.
\]

Applying Lemma~\ref{psiaf}, we get 
\begin{align*}
H_{l_1,l_2}(a) & = 
\left. \frac{\partial^2}{\partial z_1 \partial \overline{z_2}} \right|_{0}
f(a+B_2+B_1)  \\
& = 
\left.\frac{i}{4}\left( \frac{\partial^2}{\partial x_1 \partial y_2}-  \frac{\partial^2}{\partial y_1 \partial x_2}\right) \right|_{0}
f(a+(y_2x_1-y_1x_2)\alpha) \\
&=\frac{i}{2}(Df)_a(\alpha) \\
\intertext{where $(Df)_a$ denotes the differential of $f$ at $a$, so}
H_{l_1,l_2}(a)&=\frac{i}{2}\left< \alpha, \nabla f(a)\right>. 
\qedhere
\end{align*}
\end{proof}

\begin{lem}
Suppose now that $l_1=l_2=e_{\alpha}+\theta(e_{\alpha})$, then 
\[
H_{l_1,l_2}(a)= \frac{1}{2}\left<\alpha ,\nabla f(a)\right>\coth\left<\alpha , a\right>.
\]
\end{lem}

\begin{proof}
Using Lemma~\ref{decomposition}, we write
\[
il_2 = il_1  =-(\coth\left<\alpha , a\right>)l_1' + 
		\frac{1}{\sinh\left<\alpha , a\right>}\mathrm{Ad}(\exp a)(l_1')
\]
with
\[ 
l_2'=l_1'       = ie_{\alpha}-i\theta(e_{\alpha}).
\]
Then we have 
\begin{align*}
A_1 & = (x_1+x_2)l_1 -(y_1\coth\left<\alpha , a\right>+y_2\coth\left<\alpha , a\right>)l_1'   \\
B_1 & = 0\\
C_1 & = y_1\frac{1}{\sinh\left<\alpha , a\right>}\mathrm{Ad}(\exp a)(l_1')+y_2\frac{1}{\sinh\left<\alpha , a\right>}\mathrm{Ad}(\exp a)(l_2')\\
D_1 & =  [C_1,A_1]/2\\
 & =  [(y_1+y_2)il_1
		+(y_1\cosh(\left<\alpha , a\right>)
			+y_2\cosh(\left<\alpha , a\right>))l_1' , A_1]/2
\end{align*}

Once again the bracket of 
$(y_1\cosh\left<\alpha , a\right>+y_2\cosh\left<\alpha , a\right>)l_1'$ with $A_1$ yields 
only terms in $\mathfrak{k}$, so we just compute 
\[
[(y_1+y_2)il_1,A_1]=-(y_1+y_2)^2(\coth\left<\alpha , a\right>)[il_1,l_1'].
\]
Using the explicit choices of $l_1$ we have 
\begin{align*}
[il_1,l_1'] & = [i(e_{\alpha}+\theta(e_{\alpha})),ie_{\alpha}-i\theta(e_{\alpha})]\\
	& = -2\alpha.
\end{align*}
Finally we have 
\[
B_2=(y_1+y_2)^2(\coth\left<\alpha , a\right>)\alpha.
\]

Applying Lemma~\ref{psiaf}, we get 
\begin{align*}
H_{l_1,l_2}(a) & = 
\left. \frac{\partial^2}{\partial z_1 \partial \overline{z_2}} \right|_{0}
f(a+B_2+B_1)  \\
& = 
\left.\frac{1}{4}\frac{\partial^2}{\partial y_1 \partial y_2} \right|_{0}
f\left(a+(y_1^2+2y_1y_2+y_2^2)(\coth\left<\alpha , a\right>\right)\alpha) \\
&=\frac{\coth\left<\alpha , a\right>}{2}(Df)_a(\alpha) \\
&=\frac{\coth\left<\alpha , a\right>}{2}\left< \alpha, \nabla f(a)\right>.
\qedhere
\end{align*}
\end{proof}

The last step is to compute the coefficient of the Hessian with $l_1=l_2=ie_{\alpha}-i\theta(e_{\alpha})$, 
and the result is exactly the same as in the previous case:

\begin{lem}
Assume that $l_1=l_2=ie_{\alpha}-i\theta(e_{\alpha})$, then 
\[
H_{l_1,l_2}(a)= \frac{1}{2}\left<\alpha ,\nabla f(a)\right>\coth\left<\alpha , a\right>.
\]
\end{lem}

\begin{proof}
Using Lemma~\ref{decomposition}, we write
\[
il_2 = il_1 =-(\coth\left<\alpha , a\right>)l_1' + 
		\frac{1}{\sinh\left<\alpha , a\right>}\mathrm{Ad}(\exp a)(l_1')
\]
with $l_2'=l_1'  = -e_{\alpha}-\theta(e_{\alpha})$.

The beginning of the computation doesn't change:
we have 
\begin{align*}
A_1 & = (x_1+x_2)l_1 -(y_1\coth\left<\alpha , a\right>+y_2\coth\left<\alpha , a\right>)l_1'   \\
B_1 & = 0\\
C_1 & = y_1\frac{1}{\sinh\left<\alpha , a\right>}\mathrm{Ad}(\exp a)(l_1')+y_2\frac{1}{\sinh \left<\alpha , a\right>}\mathrm{Ad}(\exp a)(l_2')\\
D_1 &  =  [(y_1+y_2)il_1
		+(y_1\cosh\left<\alpha , a\right>
			+y_2\cosh\left<\alpha , a\right>)l_1' , A_1]/2
\end{align*}

Once again the bracket of 
$(y_1\cosh\left<\alpha , a\right>+y_2\cosh\left<\alpha , a\right>)l_1'$ with $A_1$ yields 
only terms in $\mathfrak{k}$, so we just compute 
\[
[(y_1+y_2)il_1,A_1]=-(y_1+y_2)^2(\coth\left<\alpha , a\right>)[il_1,l_1'].
\]

Now the expression of $l_1$ has changed, but we have 
\begin{align*}
[il_1,l_1'] & = [i(ie_{\alpha}-i\theta(e_{\alpha})),-e_{\alpha}-\theta(e_{\alpha})]\\
	& = 2[e_{\alpha},\theta(e_{\alpha})] \\
	& = -2\alpha.
\end{align*}
In other words we have again
\[
B_2=(y_1+y_2)^2(\coth\left<\alpha , a\right>)\alpha,
\]
and applying Lemma~\ref{psiaf}, we get again
\begin{align*}
H_{l_1,l_2}(a) & = 
\left. \frac{\partial^2}{\partial z_1 \partial \overline{z_2}} \right|_{0}
f(a+B_2+B_1)  \\
&=\frac{\coth\left<\alpha , a\right>}{2}\left<\alpha ,\nabla f (a)\right>.
\qedhere
\end{align*}
\end{proof}

\section{Fano group compactifications}
\label{sec_group_comp}

We present here in a concise way group compactifications and the associated polytopes,
then recall results from \cite{DelLCT} about the convex potential of a hermitian metric, 
giving more precise statements for smooth metrics. We then combine this with the computation
of the complex Monge-Ampère (Corollary~\ref{MAC})
to partially recover the formula of Brion and Kazarnovskii for the degree of an ample 
line bundle on a group compactification.
The reader willing to learn more details about group compactifications can consult, for 
example, \cite{AB04II, AK05} or \cite[Chapter 6]{BK05}. For more general references about 
spherical varieties, see \cite{Tim11, Per14}.

\subsection{Definition and examples}

\begin{defn}
We say that a projective manifold $X$ is a \emph{compactification of $G$} (or a \emph{group compactification}, when $G$ 
is not specified) if $X$ admits a holomorphic $G\times G$ action with an open and dense orbit isomorphic to $G$ as a 
$G\times G$-homogeneous space. 
We say that $(X,L)$ is a \emph{polarized compactification of $G$} if $X$ is a compactification of $G$ and $L$ is a 
$G\times G$-linearized ample line bundle on $X$.
\end{defn}

Recall that a \emph{$G$-linearized line bundle} over a $G$-manifold $X$ is a 
line bundle $L$ on $X$ equipped with an action of $G$ lifting the action on $X$, 
and such that the morphisms between the fibers induced by this action 
are linear. 

\begin{exa}
If $G=T \simeq (\mathbb{C}^*)^n$ is a torus, then the compactifications of $T$ are the  
projective toric manifolds. One goes from the $T$-action to the $T\times T$ action 
through the morphism $T\times T \rightarrow T, (t_1,t_2)\mapsto t_1t_2^{-1}$.
\end{exa}

\begin{exa}
Assume that $G$ is an adjoint semisimple group, that is, a reductive group 
with trivial center. 
Then De Concini and Procesi \cite{DCP83} proved the existence of a special compactification
of $G$, called the wonderful compactification of $G$.
It is the only compactification of $G$ satisfying the following property: its boundary 
$X\setminus G$ (where we identify 
the open and dense orbit in $X$ with $G$) is a union of simple normal crossing prime 
divisors $D_i$, $i\in \{1,\ldots, r\}$, such that for any subset 
$I\subset \{1,\ldots, r\}$,
the intersection $X\cap \bigcap_{i\in I} D_i$ is the closure of a unique $G\times G$-orbit, and 
all $G\times G$-orbits appear this way. 
The integer $r$ is equal to the rank of $G$, which is the dimension of a maximal 
torus in $G$.
\end{exa}

\subsection{Polytopes associated to a polarized group compactification}

\begin{thm}{\cite[Section 2]{AB04II}}   
Let $(X,L)$ be a  polarized group compactification of $G$. Denote by $Z$ the 
closure of $T$ in $X$. Then $Z$ is a toric manifold, equipped with a $W$-action, 
and $L|_Z$ is a $W$-linearized ample toric line bundle on $Z$.
\end{thm}

We denote by $P(X,L)$, or $P$ for simplicity, the polytope associated to the ample 
toric line bundle $L|_Z$ by the theory of toric varieties \cite{Ful93,Oda88}. 
The polytope $P$ is a lattice polytope in $M\otimes \mathbb{R}$, that we identified with
$\mathfrak{a}$, and it is $W$-invariant. 
Define $P^+(X,L):=P(X,L)\cap \overline{\mathfrak{a}^+}$. It is a polytope in 
$\mathfrak{a}$, and $P(X,L)$ is the union of the images of $P^+(X,L)$ by $W$. 

The polytope $P^+(X,L)$ encodes the structure of $G\times G$-representation of the 
space of holomorphic sections of $L$, generalizing the same property for toric line 
bundles.
\begin{prop}{\cite[Section 2.2]{AB04II}}    
Let $(X,L)$  be a polarized group compactification, then 
\[
H^0(X,L)\simeq \bigoplus \{  \mathrm{End}(V_{\alpha}) ~;~ \alpha \in M\cap P^+(X,L)\}
\]
where $V_{\alpha}$ is an irreducible representation of $G$ with highest weight $\alpha$.
\end{prop}

\begin{exa}{\cite[Proposition 6.1.11]{BK05}}
The wonderful compactification $X$ of an adjoint semisimple group is Fano.
The corresponding polytope $P(X,-K_X)$ is the convex hull of the images by 
the Weyl group $W$ of the weight $2\rho + \sum_{i=1}^r\alpha_i$, where 
the $\alpha_i$ are the simple roots of $\Phi^+$, and $2\rho$ is the sum of 
all positive roots.
\end{exa} 

We will use the three following examples to illustrate our results.

\begin{exa}
The wonderful compactification $X_0$ of $\mathrm{PGL}_2(\mathbb{C})$ is especially 
simple: it is $\mathbb{P}^3$ considered as $\mathbb{P}(\mathrm{Mat}_{2,2}(\mathbb{C}))$
equipped with the action of $\mathrm{PGL}_2(\mathbb{C})\times \mathrm{PGL}_2(\mathbb{C})$ 
induced by the multiplication of matrices on the left and on the right.

If we identify $\mathfrak{a}$ with $\mathbb{R}$ by sending the 
only positive root to $1$, then the lattice of characters is $(1/2)\mathbb{Z}$, the 
polytope $P$ associated to the anticanonical line bundle is $[-2,2]$ and the 
polytope $P^+$ is $[0,2]$.
Remark that $X_0$ is in fact homogeneous under a bigger group. This is in general not the case 
for wonderful compactifications, as the next example shows.
\end{exa}

\begin{exa}
Consider the adjoint semisimple group $\mathrm{PGL}_3(\mathbb{C})$. Then its wonderful 
compactification, that we denote by $X_1$, is Fano. Figure~\ref{PX1} represents the polytope 
$P(X_1,-K_{X_1})$, along with the Weyl chambers decomposition, 
and the weight lattice $M$. 

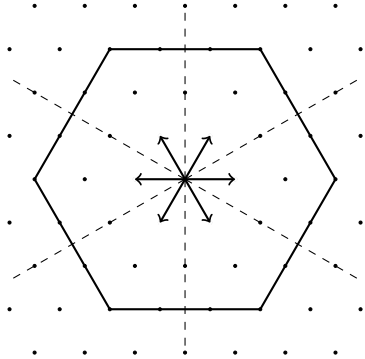
\begin{figure}
\centering
\begin{tikzpicture}[scale=2/3]
\foreach \l in {-2,-1,...,2}
{
\foreach \k in {-3,-2,...,3}
	{\draw[black, fill=black] (\k,\l*1.73) circle (0.03);}}
\foreach \l in {-3,-1,...,3}
{
\foreach \k in {-7,-5,...,7}
	{\draw[black, fill=black] (\k/2,\l*1.73/2) circle (0.03);}}
\draw[white, fill=white] (1,0) circle (0.04);
\draw[white, fill=white] (1/2,1.73/2) circle (0.04);
\draw[white, fill=white] (-1/2,1.73/2) circle (0.04);
\draw[white, fill=white] (-1,0) circle (0.04);
\draw[white, fill=white] (-1/2,-1.73/2) circle (0.04);
\draw[white, fill=white] (1/2,-1.73/2) circle (0.04);
\draw[->,thick] (0,0) -- (1,0);
\draw[->,thick] (0,0) -- (1/2,1.73/2);
\draw[->,thick] (0,0) -- (-1/2,1.73/2);
\draw[->,thick] (0,0) -- (-1,0);
\draw[->,thick] (0,0) -- (-1/2,-1.73/2);
\draw[->,thick] (0,0) -- (1/2,-1.73/2);
\draw[thick] (3,0) -- (3/2,3*1.73/2);
\draw[thick] (3/2,3*1.73/2) -- (-3/2,3*1.73/2);
\draw[thick] (-3/2,3*1.73/2) -- (-3,0);
\draw[thick] (-3,0) -- (-3/2,-3*1.73/2);
\draw[thick] (-3/2,-3*1.73/2) -- (3/2,-3*1.73/2);
\draw[thick] (3/2,-3*1.73/2) -- (3,0);
\draw[dashed] (0,0) -- (7/2,1.73*7/6);
\draw[dashed] (0,0) -- (0,1.73*2);
\draw[dashed] (0,0) -- (-7/2,1.73*7/6);
\draw[dashed] (0,0) -- (-7/2,-1.73*7/6);
\draw[dashed] (0,0) -- (0,-1.73*2);
\draw[dashed] (0,0) -- (7/2,-1.73*7/6);
\end{tikzpicture}
\caption{Polytope $P(X_1,-K_{X_1})$}
\label{PX1}
\end{figure}

Brion computed the automorphism group $\mathrm{Aut}(X)$ of all wonderful compactifications $X$ of adjoint semisimple 
groups \cite[Example 2.4.5]{Bri07}. 
In particular, the connected component of the identity $\mathrm{Aut}^0(X_1)\subset \mathrm{Aut}(X_1)$
is the image of $\mathrm{PGL}_3(\mathbb{C})\times \mathrm{PGL}_3(\mathbb{C})$.
This shows that the manifold $X_1$ is not homogeneous under a bigger group, and not toric.
The same statement is true for any wonderful compactification of an adjoint semisimple 
group with no factor of rank one.
\end{exa}

\begin{exa}
Consider now the simply connected semisimple group $\mathrm{Sp}_4(\mathbb{C})$. 
It is the subgroup of $\mathrm{GL}_4(\mathbb{C})$ whose elements are the matrices 
$M$ such that $M^t\Omega M= \Omega$, where $M^t$ denotes the transpose of $M$ 
and 
\[
\Omega = 
\begin{pmatrix}
0 & 1_n \\
1_n & 0 
\end{pmatrix}.
\]
It is a semisimple group of type $B_2$, which means that its associated root 
system is the root system $B_2$. 
While it is not an adjoint group, it also admits a wonderful compactification 
(see for example \cite{GR13}). 
Let $X_2$ be the blow up of this wonderful compactification along the closed orbit. 
Then $X_2$ is a Fano compactification of $\mathrm{Sp}_4(\mathbb{C})$ \cite{Ruz12}.

The moment polytope associated to the anticanonical line bundle is obtained easily 
in the following way: Ruzzi classifies 
in \cite{Ruz12} all Fano compactifications of reductive groups of rank two, 
by giving for each the usual data identifying a spherical variety. 
Then we can use the description by Gagliardi and Hofscheier \cite{GH15} 
of the moment polytope of the anticanonical line bundle in these terms. 
Figure~\ref{PX2} represents the polytope $P(X_2,-K_{X_2})$, along with the Weyl chambers 
decomposition, and the weight lattice $M$. 

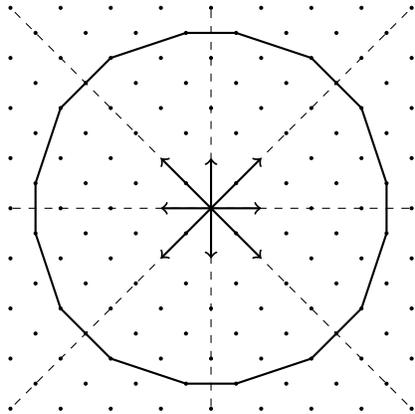
\begin{figure}
\centering
\begin{tikzpicture}[scale=2/3]
\foreach \l in {-4,-3,...,4}
{
\foreach \k in {-4,-3,...,4}
	{\draw[black, fill=black] (\k,\l) circle (0.03);}}
\foreach \l in {-7,-5,...,7}
{
\foreach \k in {-7,-5,...,7}
	{\draw[black, fill=black] (\k /2,\l /2) circle (0.03);}}
\draw[white, fill=white] (1,0) circle (0.04);
\draw[white, fill=white] (1,1) circle (0.04);
\draw[white, fill=white] (0,1) circle (0.04);
\draw[white, fill=white] (-1,1) circle (0.04);
\draw[white, fill=white] (-1,0) circle (0.04);
\draw[white, fill=white] (-1,-1) circle (0.04);
\draw[white, fill=white] (0,-1) circle (0.04);
\draw[white, fill=white] (1,-1) circle (0.04);
\draw[->,thick] (0,0) -- (1,0);
\draw[->,thick] (0,0) -- (1,1);
\draw[->,thick] (0,0) -- (0,1);
\draw[->,thick] (0,0) -- (-1,1);
\draw[->,thick] (0,0) -- (-1,0);
\draw[->,thick] (0,0) -- (-1,-1);
\draw[->,thick] (0,0) -- (0,-1);
\draw[->,thick] (0,0) -- (1,-1);
\draw[thick] (2,3) -- (1/2,7/2) -- (-1/2,7/2) -- (-2,3) -- (-3,2) -- (-7/2,1/2);
\draw[thick] (-7/2,1/2) -- (-7/2,-1/2) -- (-3,-2) -- (-2,-3) -- (-1/2,-7/2);
\draw[thick] (-1/2,-7/2) -- (1/2,-7/2) -- (2,-3) -- (3,-2) -- (7/2,-1/2);
\draw[thick] (7/2,-1/2) -- (7/2,1/2) -- (3,2) -- (2,3);
\draw[dashed] (0,0) -- (4,0);
\draw[dashed] (0,0) -- (4,4);
\draw[dashed] (0,0) -- (0,4);
\draw[dashed] (0,0) -- (-4,4);
\draw[dashed] (0,0) -- (-4,0);
\draw[dashed] (0,0) -- (-4,-4);
\draw[dashed] (0,0) -- (0,-4);
\draw[dashed] (0,0) -- (4,-4);
\end{tikzpicture}
\caption{Polytope $P(X_2,-K_{X_2})$}
\label{PX2}
\end{figure}

For this manifold also, $\mathrm{Aut}^0(X_2)$ is the image of 
$\mathrm{Sp}_4(\mathbb{C})\times \mathrm{Sp}_4(\mathbb{C})$.
This is proved using Blanchard's lemma \cite[Proposition I.1]{Bla56} applied to the blow up, 
and the fact that the connected component of the identity in the automorphism group of 
the wonderful compactification of $\mathrm{Sp}_4(\mathbb{C})$ is also the image of 
$\mathrm{Sp}_4(\mathbb{C})\times \mathrm{Sp}_4(\mathbb{C})$ \cite{Pez09}. 
\end{exa}

\subsection{Convex potentials}  

Let $(X,L)$ be a polarized compactification of $G$. 
We recall in this section how to associate to a $K\times K$-invariant hermitian metric 
with positive curvature on $L$ a 
convex function, called its convex potential (see \cite[Section 2.2]{DelLCT}). 
We let $P:=P(X,L)$, $P^+:=P^+(X,L)$, and let $v$ denote the support function of the 
dilated polytope $2P$, defined by $v(x) = \mathrm{sup}\{\left<p,x\right>~;~ p\in 2P\}.$

Let $s$ be a $G\times \{e\}$-equivariant trivialization of $L|_G$, where we identify $G$ with the 
open $G\times G$-orbit in $X$.
Let $h$ be a smooth, $K\times K$-invariant, positively curved hermitian metric on $L$. 
 Let $\varphi$ be the local potential of $h$ with respect to the 
trivialization $s$, defined for $g\in G$ by: 
\[
\varphi(g):=-\ln|s(g)|_h^2.
\]
Then the function $\varphi$ is a smooth strictly plurisubharmonic function on $G$, 
$K\times K$-invariant 
by \cite[Proposition 2.2]{DelLCT}.

The function $u$ defined on $\mathfrak{a}$ by $u(x)=\varphi(\exp x)$ is then called the 
\emph{convex potential} 
of $h$ and satisfies the following properties.

\begin{prop}
\label{Potentials}
Let $h$ be a smooth $K\times K$-invariant hermitian metric with positive curvature on $L$, and let 
$u$ be its convex potential. Then $u$ is a $W$-invariant, smooth and strictly convex function 
such that:
\begin{enumerate}
\item $u(x)\geq v(x)+C_1$ for all $x\in \mathfrak{a}$, for some constant $C_1$ depending on $u$;
\item given any $x_0\in \mathfrak{a}$, $u(x)\leq v(x-x_0)+u(x_0)$ for all $x\in \mathfrak{a}$;
\item the gradient $\nabla u$ of $u$ defines a diffeomorphism from $\mathfrak{a}$ to the interior of $2P$;
\item in particular, there exists a constant $d$ depending only on $L$ such that $|\nabla u(x)|\leq d$ for all $x\in \mathfrak{a}$; 
\item the restriction of  $\nabla u$ to $\mathfrak{a}^+$ is a diffeomorphism to the interior of $2P^+$.
\end{enumerate}
\end{prop}

\begin{proof}
The fact that $u$ is a $W$-invariant smooth and strictly convex function follows from its definition and \cite{AL92}. 

A smooth metric has locally bounded potentials, so \cite[Theorem 2.4]{DelLCT} 
implies that there exists constants $C_1$ and $C_2$ 
depending on $u$ such that 
\[
v(x)+C_1\leq u(x) \leq v(x)+C_2.
\]
In particular, the first point is proved. Let us prove the second point.

Let $x_0\in \mathfrak{a}$. For any $0\neq y\in \mathfrak{a}$, consider the slope 
\[
\frac{u(x_0+ty)-u(x_0)}{t},
\] 
with $t>0$.
By convexity and the two previous inequalities, we see that this 
slope increases and converges to $v(y)$ as $t$ tends to infinity. 
This shows that for any $x=x_0+y \in \mathfrak{a} \setminus \{x_0\}$, we have 
\[
u(x)\leq v(x-x_0)+u(x_0).
\]
This inequality is obviously also satisfied at $x_0$, so the second point is proved.

Since $u$ is smooth and strictly convex, $\nabla u$ is a diffeomorphism. 
It remains to determine its image. 
It is clear that this image is open and included in $2P$ (for example because the domain of 
the convex conjugate of $u$ is $2P$, see \cite[Proposition 3.12]{DelLCT}). 
Assume now that one point from the interior of $2P$ is not attained by $\nabla u$.
Then by convexity, we can find a $0\neq y\in \mathfrak{a}$ and an $\epsilon>0$ 
such that $\forall x\in \mathfrak{a}$,
$v(y) - \left<\nabla u(x),y\right> > \epsilon$. 
But considering the ray starting from any $x_0$ and 
in the direction $y$ as above leads to a contradiction. 
This implies that $\nabla u (\mathfrak{a})= \mathrm{Int}(2P)$. 
By $W$-invariance, we also have $\nabla u (\mathfrak{a}^+)= \mathrm{Int}(2P^+)$.

Finally, the fourth point is a direct consequence of the third because the polytope $2P$ is bounded.
\end{proof}

\begin{rem}
We fixed a choice of a $G\times \{e\}$-equivariant trivialization of $L$ on $G$ to define the convex 
potential. All such trivializations are non zero scalar multiples of one another, so the convex potentials
defined with respect to two different such trivializations differ by a constant scalar independent of the metric.
In particular, it does not change the asymptotic behavior. 
\end{rem}

\subsection{Degree of an ample linearized line bundle}

\begin{prop}
\label{degrees}
Let $(X,L)$ be a polarized compactification of $G$, and let $P^+=P^+(X,L)$.
Then 
\[
\mathrm{deg}(L)=C\int_{2P^+} \prod_{\alpha \in \Phi^+}\left<\alpha , p\right>^2dp
\]
for some constant $C$ depending only on the group $G$.
Furthermore, if $u$ is the convex potential of a smooth positively curved $K\times K$-invariant 
metric on $L$, then 
\[
\mathrm{deg}(L)=
C\int_{\mathfrak{a}^+} \prod_{\alpha \in \Phi^+}\left<\alpha ,\nabla u(a)\right>^2
\mathrm{MA}_{\mathbb{R}}(u)(a)da.
\]
\end{prop}

\begin{proof}
Let $h$ be a smooth positively curved $K\times K$-invariant hermitian metric on $L$, 
with curvature the Kähler form $\omega \in 2\pi c_1(L)$. Then we have 
\[
\mathrm{deg}(L)  = \int_X \left(\frac{\omega}{2\pi} \right)^n. 
\]

Let $s$ be a $G\times \{e\}$-equivariant section of $L$, and $\varphi$ the potential 
of $h$ with respect to $s$. We have $\omega = i \partial \overline{\partial} \varphi$ 
on $G$. 
Using Remark~\ref{ChoiceofHaar}, Choose a Haar volume form $dg$ on $G$ so that, 
independently of $\varphi$,
\[
\left(\frac{i}{2\pi}\partial \overline{\partial} \varphi\right)^n =
 \mathrm{MA}_{\mathbb{C}}(\varphi) dg,
\]
where $ \mathrm{MA}_{\mathbb{C}}(\varphi)$ is the Monge-Ampère of $\varphi$ in the coordinates 
chosen in Section~\ref{CHess}. 
Then 
\begin{align*}
\mathrm{deg}(L) & = \int_X \left(\frac{\omega}{2\pi} \right)^n  = \int_G \left(\frac{\omega}{2\pi} \right)^n \\
& = \int_G \left(\frac{i}{2\pi} \partial \overline{\partial} \varphi \right)^n \\
& = \int_G \mathrm{MA}_{\mathbb{C}}(\varphi) dg. \\
\intertext{By the $KAK$ integration formula, this is, for a constant $C$ depending only on the choice of Haar measure,}
\mathrm{deg}(L) & = C \int_{\mathfrak{a}^+}  \mathrm{MA}_{\mathbb{C}}(\varphi)(\exp a) J(a)da \\
\intertext{where $da$ is the Lebesgue measure on $\mathfrak{a}$ normalized by $N$. 
Let $u$ denote the convex potential of $h$.
From Corollary~\ref{MAC} we obtain that this is}  
\mathrm{deg}(L) & = C' \int_{\mathfrak{a}^+} \prod_{\alpha\in \Phi^+}\left<\alpha ,\nabla u(a)\right>^2 \mathrm{MA}_{\mathbb{R}}(u)(a)  da  \\
\intertext{for a constant $C'$ still depending only on $G$. We then use the variable change $p=\nabla u(a)$, and obtain, by 
Proposition~\ref{Potentials}, }
\mathrm{deg}(L) & = C' \int_{2P^+}  \prod_{\alpha\in \Phi^+}\left<\alpha , p\right>^2dp.
\qedhere
\end{align*}
\end{proof}

\section{Strategy of proof}
\label{sec_strategy}

We describe in this section the global strategy of the proof, starting by recalling 
classical results on the continuity method. We then determine the expression of the 
continuity method equation in restriction to $G$, in terms of convex potentials. 
We finally introduce the function $\nu_t$ and gather some estimates on this function 
or on $-\ln J$ that will be used later.

\subsection{The continuity method}
Let $X$ be a Fano manifold.
Fix a reference Kähler form $\omega_{\mathrm{ref}}$ in the class $2\pi c_1(X)$.
The Kähler forms in $2\pi c_1(X)$ can all be written as $\omega_{\mathrm{ref}}+i\partial \overline{\partial} \psi$ 
with $\psi$ a smooth and $\omega_{\mathrm{ref}}$-strictly psh function on $X$ (i.e. such that 
$\omega_{\mathrm{ref}}+i\partial \overline{\partial} \psi >0$).

The Kähler-Einstein equation  $\mathrm{Ric}(\omega)=\omega$ on $X$ translates, 
in terms of $\omega_{\mathrm{ref}}$-psh functions, as
the Monge-Ampère equation
\begin{equation}
\label{KEeqn}
(\omega_{\mathrm{ref}} + i\partial \overline{\partial} \psi)^n=e^{f_{\mathrm{ref}}-\psi}\omega_{\mathrm{ref}}^n,
\end{equation}
where $f_{\mathrm{ref}}$ is the \emph{normalized Ricci potential} of $\omega_{\mathrm{ref}}$ defined as the $\omega_{\mathrm{ref}}$-psh
function that satisfies $\omega_{\mathrm{ref}}+i\partial \overline{\partial}f_{\mathrm{ref}} = \mathrm{Ric}(\omega_{\mathrm{ref}})$
and $\int_X e^{f_{\mathrm{ref}}}\omega_{\mathrm{ref}}^n 
= \int_X\omega_{\mathrm{ref}}^n$.

Let $h_{\mathrm{ref}}$ be a smooth hermitian metric on $-K_X$ with curvature form 
$\omega_{\mathrm{ref}}$. Then it determines a volume form $dV$ on $X$ defined in a local
trivialization $s$ of $-K_X$ by $dV=|s|_{h_{\mathrm{ref}}}^2s^{-1}\wedge \overline{s^{-1}}$.
Then up to a constant, the Ricci potential $f_{\mathrm{ref}}$ is equal to the logarithm
of the potential of $dV$ with respect to $\omega_{\mathrm{ref}}^n$. 
We choose $h_{\mathrm{ref}}$ (by multiplying by a scalar) such that $f_{\mathrm{ref}}$ 
is indeed equal to that.

The following family of equations is the one used in the usual continuity method for the 
Kähler-Einstein equation:
\begin{equation}
\label{Conteqn}
(\omega_{\mathrm{ref}}+i\partial  \overline{\partial} \psi_t)^n = 
e^{f_{\mathrm{ref}}-t\psi_t}\omega_{\mathrm{ref}}^n.
\end{equation}

To show the existence of a Kähler-Einstein metric on $X$, it is enough to show that 
the set $I$ of $0\leq t\leq 1$ such that this equation admits a solution is exactly 
$\left[0,1\right]$.

By the work of Aubin \cite{Aub76} 
and Yau \cite{Yau78}, $0\in I$, and $I$ is open.
Furthermore,  it is enough to know uniform a priori estimates on the $C^0$ norm 
of $\psi_t$, to ensure the closure of $I$, and thus the existence of a solution at $t=1$, 
i.e. a Kähler-Einstein metric.
We recall that by $C^0$ estimates, we mean, as in most of the literature, 
a uniform control on $\mathrm{sup}_{x\in X}|\psi_t(x)|$.
In fact, we can even concentrate only on a uniform upper bound on 
$\psi_t$ (see \cite[Proposition 2.1]{Siu88} or \cite[pages 235 and 236]{Tia87}).

\begin{nota}
\label{notat0}
Let us fix some $0<t_0\in I$, which exists since $0\in I$ and $I$ is open.  
\end{nota}
 
Let us summarize the consequence of what we have recalled in this section.

\begin{prop}
\label{psiest}
Assume that $[t_0,t_1[\subset I$, that $\psi_t$ denotes the solution at $t\in [t_0,t_1[$,  
and that there exists a constant $C$ such that 
$\psi_t\leq C$ $\forall t\in [t_0,t_1[$.
Then $t_1\in I$.
\end{prop}

\subsection{Reduction to the open orbit}

Suppose now that $X$ is a $G\times G$-equivariant smooth and Fano compactification of $G$. 
Let $P$ be the polytope associated to the anticanonical bundle $-K_X$.

We choose $\omega_{\mathrm{ref}}$ and $h_{\mathrm{ref}}$ $K\times K$-invariant.
Then a solution of equation~\eqref{Conteqn} at $t<1$ is $K\times K$-invariant.
This follows from the uniqueness result for twisted (or generalized) 
Kähler-Einstein metrics (see \cite[Corollary 1.4]{ZZ14} or \cite{DS15}).  

By continuity of the solutions $\psi_t$, it is enough to prove a uniform upper bound on the 
restrictions of $\psi_t$ to the open and dense orbit $G\subset X$.
Let $\varphi_t$ denote the function on $\mathfrak{a}$ induced by $\psi_t$. 
It is enough to 
give an upper bound for $\varphi_t$.
We also denote by $h_t$ the hermitian metric on $-K_X$ whose potential with respect 
to $h_{\mathrm{ref}}$ is $\psi_t$ and whose curvature form is 
$\omega_{\mathrm{ref}}+i\partial  \overline{\partial} \psi_t$.

Any $G\times \{e\}$-equivariant trivialization $s$ of $-K_G$ defines a left $G$-invariant
volume form $s^{-1} \wedge \overline{s^{-1}}$ on $G$, and so a Haar volume form.
All Haar volume forms on $G$ are obtained this way. 
By Remark~\ref{ChoiceofHaar}, we can thus choose $s$ and the corresponding Haar volume form 
$dg:= s^{-1} \wedge \overline{s^{-1}}$ such that 
for all smooth $K\times K$-invariant function $\psi$ on $G$, 
\[
(i\partial \overline{\partial} \psi)^n 
= \mathrm{MA}_{\mathbb{C}}(\psi)dg,
\]
where $\mathrm{MA}_{\mathbb{C}}(\psi)$ is the complex Monge-Ampère of $\psi$ 
in the local coordinates given in Section~\ref{CHess}.

Let $u_{\mathrm{ref}}$ be the convex potential of $h_{\mathrm{ref}}$ with respect to 
the trivialization $s$, defined on $\mathfrak{a}$, and denote by   
$u_t$ the convex function $u_{\mathrm{ref}}+\varphi_t$ 
which is the convex potential of the metric $h_t$. 
Finally, we denote by $w_t$ the function $tu_t+(1-t)u_{\mathrm{ref}}$. 

\begin{prop}
\label{eqnG}
Suppose $\psi_t$ is a  $K\times K$-invariant solution of equation \eqref{Conteqn}. 
Then for $x\in \mathfrak{a}$,
\begin{equation}
\mathrm{MA}_{\mathbb{R}}(u_t)(x)\prod_{\alpha \in  \Phi^+}\left<\alpha , \nabla u_t(x)\right>^2 
= e^{-w_t(x)}J(x).
\end{equation}
\end{prop}

\begin{proof}
Let $\psi_{\mathrm{ref}}$ be the potential of the reference metric $h_{\mathrm{ref}}$ with 
respect to 
the section $s$.
By the choice of $s$, we have, on $G$,
\begin{align*}
(\omega_{\mathrm{ref}}+i\partial  \overline{\partial} \psi_t)|_G^n & = (i\partial \overline{\partial}(\psi_{\mathrm{ref}} + \psi_t))^n \\
 & = \mathrm{MA}_{\mathbb{C}}(\psi_{\mathrm{ref}} + \psi_t)dg
\end{align*}

Now the computation of the complex Monge-Ampère in local coordinates from 
Section~\ref{CHess}
gives 
\[
(\omega_{\mathrm{ref}}+\partial  \overline{\partial} \psi_t)^n(\exp x)  = 
	\mathrm{MA}_{\mathbb{R}}(u_t)(x)
		\prod_{\alpha \in  \Phi^+}\left<\alpha , \nabla u_t(x)\right>^2 
			\frac{1}{J(x)} dg 
\]
for $x\in \mathfrak{a}^+$.

On the other hand, the definition of the normalized Ricci potential 
$f_{\mathrm{ref}}$ implies 
that 
\[
e^{f_{\mathrm{ref}}}\omega_{\mathrm{ref}}^n=e^{-\psi_{\mathrm{ref}}} dg,
\]
which allows to write the right hand side of equation \eqref{Conteqn} as 
\[ 
e^{f_{\mathrm{ref}}-t\psi_t}\omega_{\mathrm{ref}}^n  = 
e^{-t\psi_t -\psi_{\mathrm{ref}}} dg. 
\]
For $x\in \mathfrak{a}$, we have 
\begin{align*}
(-t\psi_t -\psi_{\mathrm{ref}})(\exp x) & =-t\varphi_t(x) - u_{\mathrm{ref}}(x) \\
 & = -tu_t(x)-(1-t)u_{\mathrm{ref}}(x) \\
 & = -w_t(x).
\end{align*}

In conclusion, at a point $\exp(x)$ for $x\in \mathfrak{a}^+$, equation \eqref{Conteqn} 
reads
\[
\mathrm{MA}_{\mathbb{R}}(u_t)(x) \prod_{\alpha \in  \Phi^+}\left<\alpha , \nabla u_t(x)\right>^2 
		\frac{1}{J(x)} dg  = 
e^{-w_t(x)} dg.
\]
It is equivalent to the equality of the potentials with respect to $dg$.
Furthermore, by multiplying both sides by $J(x)$, we obtain the equation of the statement, 
that is well defined 
on the whole of $\mathfrak{a}$, and it is satisfied by $W$-invariance and smoothness.
\end{proof}

\subsection{Strategy}

To find a uniform upper bound for $\varphi_t$ we will introduce another function $\nu_t$, 
and study this function, 
following the strategy used by Wang and Zhu in the toric case.
More precisely, let $j$ be the function on the open Weyl chamber $\mathfrak{a}^+$
defined by $j(x)=-\ln J(x)$.
We consider the function $\nu_t= w_t+j$ defined on $\mathfrak{a}^+$.
We will show that it is a strictly convex function on $\mathfrak{a}^+$.
It is proper in the following sense:
As $x$ goes to infinity, or $x$ goes to a wall of $\mathfrak{a}^+$, 
$\nu_t(x)$ goes to infinity. 

These two properties of $\nu_t$ imply that it admits a unique minimum.
Let $m_t$ be the minimum of $\nu_t$ and $x_t$ be the point of $\mathfrak{a}^+$
where this minimum is attained. 
We will obtain estimates on both the value $m_t$ of the minimum and on 
the distance from the origin $|x_t|$ of the point where it is attained.
Then we need to relate these estimates with the function that we want to control.
Namely we will go from $\nu_t$ to $w_t$ then $u_t$ and finally $\varphi_t$.

To summarize, the strategy to prove estimates on $\varphi_t$ is in three steps:
\begin{itemize}
\item reduce to estimates on  $|m_t|$ and $|x_t|$,
\item find uniform estimates $|m_t|\leq C$, 
\item get a uniform control $|x_t|\leq C$ of $x_t$.
\end{itemize}
We will also have to prove the necessity of 
the condition in Theorem~A
and the upper bound on $R(X)$.
Before that, we gather some preliminary results.  

\subsection{The function $j=-\ln J$}
The half sum of positive roots $\rho$ is in $\mathfrak{a}^+$, so
$\left<\alpha ,\rho\right>>0$ for all $\alpha\in \Phi^+$. 
We will use this as a reference to control the distance to the walls of $\mathfrak{a}^+$.
We also choose a basis $\{e_i\}$ of $\mathfrak{a}$, 
and corresponding coordinates $\{x_i\}$ when necessary.

\begin{lem}
\label{Dominate}
There exists a constant $c$ such that for any $x\in \mathfrak{a}^+$, we have
\[
j(x)\geq \left<-4\rho,x\right>+c.
\]
Furthermore, for any $b>0$, there exists constants $C=C(b)$, $C'=C'(b)$ such that 
for all $x\in b\rho +\mathfrak{a}^+$, 
\[
j(x)\leq \left<-4\rho,x\right>+C \leq C'.
\]
\end{lem}

\begin{proof}
Write 
\[
\sinh \left<\alpha , x\right>=e^{\left<\alpha , x\right>}\left(\frac{1-e^{-2\left<\alpha , x\right>}}{2}\right)\leq \frac{e^{\left<\alpha , x\right>}}{2}
\] 
for $x\in \mathfrak{a}^+$.
Then 
\[
j(x)=-2 \sum_{\alpha \in \Phi^+}\ln \sinh \left<\alpha , x\right>
\geq -2\sum_{\alpha \in \Phi^+}\left<\alpha , x\right>+c,
\]
where $c=2\ln(2)\mathrm{Card}(\Phi^+)$.

If we assume $x\in b\rho +\mathfrak{a}^+$, then 
using the same expression of $\sinh \left<\alpha , x\right>$, we get 
$j(x)\leq \left<-4\rho,x\right>+C $ with 
\[
C=-2\sum_{\alpha\in \Phi^+}\ln\left( \frac{1-e^{-2b\left<\alpha,\rho\right>}}{2} \right).
\]
and $j(x)\leq C'$ with 
\[
C'= C-2b\sum_{\alpha \in  \Phi^+}\left<\alpha,\rho\right>.
\qedhere
\]
\end{proof}

Let us now turn to the derivatives of $j$.
Recall that the directional derivative of $j$ in the direction $\xi$ is defined by:
\[
\frac{\partial j}{\partial \xi}(x):= \left<\nabla j(x), \xi \right>.
\]

\begin{lem}
\label{j4rho}
For any $\xi \in \mathfrak{a}^+$, we have, for all $x\in \mathfrak{a}^+$,
\[
\frac{\partial j}{\partial \xi}(x) < \left< -4\rho, \xi \right>.
\]
If $\xi \in \mathfrak{a}_{t}$, then for all $x\in \mathfrak{a}^+$,
\[
\frac{\partial j}{\partial \xi}(x)=0.
\]
\end{lem}

\begin{proof}
The second statement is clear because for any $\alpha \in  \Phi$, $\left<\alpha,x\right>$ 
depends only on the projection of $x$ on $\mathfrak{a}_{ss}$.

For the first statement, we compute 
\[
\frac{\partial j}{\partial \xi}(x) = 
	-2\sum_{\alpha\in \Phi^+}\left<\alpha,\xi\right> \coth \left<\alpha,x\right>.
\]
Since $\xi \in \mathfrak{a}^+$, we have $\left<\alpha,\xi\right> >0$ for all $\alpha \in \Phi^+$.
Since $x$ is also in $\mathfrak{a}^+$, we have $\coth\left<\alpha,x\right>>1$.
We deduce from this that 
\[
\frac{\partial j}{\partial \xi}(x) < -2\sum_{\alpha\in \Phi^+}\left<\alpha,\xi\right> 
= \left<-4\rho, \xi\right>.
\qedhere
\]
\end{proof}

We will also need to control the derivatives of $j$ away from the walls. 
This is achieved by the following lemma.

\begin{lem}
\label{jderivatives}
For any $b>0$, there exists a constant $C$ such that for any 
$x\in b\rho + \mathfrak{a}^+$, 
\[
|\nabla(j)(x)|\leq C.
\] 
\end{lem}

\begin{proof}
Recall that 
\[
\frac{\partial j}{\partial x_i}(x) = 
	-2\sum_{\alpha\in \Phi^+}\left<\alpha ,e_i\right>\coth\left<\alpha , x\right>
\]
For $x\in b\rho + \mathfrak{a}^+$, we have 
$1<\coth\left<\alpha , x\right><\coth(b\left<\alpha ,\rho\right>)$,
so for any $i$, 
\[
\left|\frac{\partial j}{\partial x_i}(x)\right| \leq 
	2\sum_{\alpha\in \Phi^+}|\alpha|\coth(b\left<\alpha ,\rho\right>).
\qedhere
\]
\end{proof}

\begin{lem}
\label{strictconvex}
The function $j$ is strictly convex on $\mathfrak{a}^+$.
\end{lem}

\begin{proof}
An easy computation shows that 
\[
\frac{\partial^2}{\partial x_j\partial x_i}(-\ln \sinh \left<\alpha , x\right>) = 
	 \frac{\left<\alpha ,e_i\right> \left<\alpha ,e_j\right>}{\sinh^2\left<\alpha , x\right>}.
\]
So the Hessian of $j$ is a sum of semipositive matrices, 
and it is easy to check that the whole sum is definite, 
so the Hessian of $j$ is positive definite.
\end{proof}

\subsection{The function $\nu_t$}

\begin{lem}
The function $\nu_t$ is strictly convex and admits a unique minimum.
\end{lem}

\begin{proof}
From Lemma~\ref{strictconvex}, and the fact that $w_t$ is strictly convex as a 
convex combination of strictly convex functions, we obtain that 
$\nu_t = w_t + j$ is strictly convex. 
The function $w_t$ is bounded below, and $j$ tends to $+\infty$ when $x$ approaches 
a Weyl wall, so $\nu_t(x)$ also tends to $+\infty$ when $x$ approaches 
a Weyl wall. To prove the existence of a minimum $x_t$ it remains to explain why 
$\nu_t$ goes to infinity at infinity. 

Proposition~\ref{Potentials} implies that 
$w_t(x)\geq v(x)+C_1$ for some constant $C_1$, where $v$ is the support function 
of the polytope $2P$, so 
$\nu_t(x)\geq v(x)+j(x)+C_1$.
Then 
\[
\nu_t(x)\geq v(x)-\left<4\rho ,x\right>+c+C_1
\]
 by Lemma~\ref{Dominate}.
Finally, the fact that $X$ is Fano implies, by \cite[Remark 4.10]{DelLCT}   
that $4\rho \in \mathrm{Int}(2P)$, so $\nu_t$ is indeed proper.  
\end{proof}

Let $m_t$ denote the minimum value of $\nu_t$ and $x_t$ be such that $\nu_t(x_t)=m_t$.

\begin{lem}
\label{translate1}
There exists a constant $b_1>0$ independent of $t$ such that 
\[
x_t \in b_1\rho + \mathfrak{a}^+
\]
\end{lem}

\begin{proof}
By definition of $x_t$, the derivative of $\nu_t$ at $x_t$ vanishes.
In particular, we have 
\[
\frac{\partial \nu_t}{\partial \rho}(x_t) = 0.
\]
Recall that $\nu_t=w_t+j$, and that the derivatives of $w_t$ are bounded by Proposition~\ref{Potentials}.
In particular we get a bound 
\[
\left|\frac{\partial w_t}{\partial \rho}(x_t)\right|\leq C.
\]
On the other hand, recall that: 
\[
\frac{\partial j}{\partial \rho}(x_t) = 
-2\sum_{\alpha\in \Phi^+}\left<\alpha ,\rho\right>\coth\left<\alpha, x_t\right>.
\]
So we have 
\[
\left|2\sum_{\alpha\in \Phi^+}\left<\alpha ,\rho\right>\coth\left<\alpha, x_t\right>\right|\leq C
\]
but since all the terms of the sum are positive and all the $\left<\alpha ,\rho\right>$ are strictly positive,
this implies that for all $\alpha \in \Phi^+$, $\coth\left<\alpha, x_t\right>\leq C$.
Observe that the function $\coth$ tends to $+\infty$ at 0, so we obtain 
$\left<\alpha, x_t\right>\geq C'$ for all $\alpha$ for a constant $C'>0$ independent of $t$.

To conclude, observe that the intersection of the half spaces defined by $\left<\alpha , x\right>\geq C'$
is contained in a translate $b_1\rho +\mathfrak{a}^+$ for some $b_1>0$
sufficiently small, independent of $t$.
\end{proof}

We will also need to control from below the value of $\nu_t$ near the walls.
This will be achieved by the following technical proposition.
For now we cannot control $\nu_t$ uniformly close to the walls, but we will as soon as 
we control $m_t$.
We will use twice the proposition, first to obtain a lower bound on $m_t$, then 
to ensure $e^{-\nu_t}$ is uniformly sufficiently small near the walls.
Remark also that this proposition can be seen as a precise statement of what we called the
properness of $\nu_t$ near the walls.

\begin{prop}
\label{technical}
For any $M>0$, there exists a constant $b>0$ independent of $t$ such that 
for any $x\in \mathfrak{a}^+$ satisfying $\left<\alpha , x\right><b\left<\alpha ,\rho\right>$ for some root 
$\alpha \in \Phi^+$
defining a wall of $\mathfrak{a}^+$, we have
\[
\nu_t(x)\geq m_t+M.
\]
\end{prop}

Recall that the roots defining the walls are also the simple roots of $\Phi^+$.

\begin{proof}
Let $x\in  \mathfrak{a}^+$ be such that $\left<\alpha , x\right><b_1\left<\alpha ,\rho\right>$ for some simple root 
$\alpha \in \Phi^+$.
Consider the ray $\{x+s\rho, s\geq 0\}$ starting from $x$. It meets the boundary
$\partial (b_1\rho + \mathfrak{a}^+)$ of $b_1\rho + \mathfrak{a}^+$
at a unique point $y=x+s_0\rho$.
Furthermore $y$ is in $b_1\rho+ \alpha^{\perp}$ for a simple root $\alpha$.
We can then write $x=y-s_0\rho$, and $s_0$ satisfies $0<b_1-s_0 <b$.

Consider $\alpha \in \Phi^+$ a simple root, and
$y\in (b_1\rho +  \alpha^{\perp})\cap \partial (b_1\rho + \mathfrak{a}^+)$.
We will show that there exists a constant $b>0$ independent of $t$ such that $\nu_t(y-s\rho)\geq m_t+M$  for
all $s$ such that $0<b_1-s <b$, and that this $b$ can be chosen independent of $y$ and 
$\alpha$.

This is enough to prove the proposition because any $x$ as in the statement is of the form 
above for some $\alpha$, $y$ and $s$ as shown at the beginning.

Consider the function 
$g(s)=\nu_t(y-s\rho)$ on $\left[0,b_1\right[$.
We have $g(0)=\nu_t(y)\geq m_t$ by definition of $m_t$.

We consider now the derivative of $g$. Remember that the derivatives of $w_t$ are
uniformly bounded, by $d$, in absolute value by Proposition~\ref{Potentials}. Then 
\[
g'(s)\geq -d+2\sum_{\beta\in \Phi^+}\left<\beta,\rho\right>\coth\left<\beta, y-s\rho\right>.
\]
Since all the terms in the sum are positive, we have in particular 
\[
g'(s)\geq -d+2\left<\alpha ,\rho\right>\coth\left<\alpha,y-s\rho\right>.
\]
From the assumptions, we compute 
\[
\left<\alpha, y-s\rho\right>=b_1\left<\alpha ,\rho\right>-s\left<\alpha ,\rho\right>=(b_1-s)\left<\alpha ,\rho\right>.
\]

Observe that the positive function $\coth$ is not integrable near $0^+$, so 
\[
\int_0^{b_1}\coth((b_1-s)\left<\alpha ,\rho\right>)ds=+\infty.
\]
Together with the fact that $g$ is greater than $m_t$ at 0, it means that 
for any $M$, we can find a $b_{\alpha}>0$ such that 
$g(s)\geq M+m_t$ for $b_1-s\leq b_{\alpha}$.

Remark that none of what we have done depends on the choice of $y$.
Furthermore, since there are only a finite number of roots $\alpha$, we can choose a
$b>0$ such that $b<b_{\alpha}$ for all $\alpha$, and it concludes the proof. 
\end{proof}

\subsection{Reduction to estimates on $m_t$ and $x_t$}

\begin{lem}
\label{lem_est_m_x}
Suppose we have uniform estimates $|m_t|<C_m$ and $|x_t|<C_x$ for $t$ in some 
interval $[t_0,t_1[ \subset [0,1]$. Then there is an uniform upper bound for 
$\varphi_t$ on $[t_0,t_1[$.
\end{lem}

\begin{proof}
Recall that it is enough to obtain a uniform upper bound on $u_t-u_{\mathrm{ref}}$ which 
is a function defined on $\mathfrak{a}$.
We have, by Proposition~\ref{Potentials} with $x_0=x_t$, that 
\[
u_t(x) \leq v(x-x_t)+u_t(x_t)
\]
where $v$ is the support function of $2P$.
Using the other inequality for $u_{\mathrm{ref}}$ we have 
\[
u_{\mathrm{ref}}(x) \geq v(x)+C_1
\geq v(x-x_t) +C_1- d|x_t|.
\]
Combining these two gives 
\begin{align*}
(u_t-u_{\mathrm{ref}})(x) & \leq v(x-x_t)+u_t(x_t)-v(x-x_t)-C_1+d|x_t| \\
& \leq u_t(x_t) -C_1+d|x_t| \\
& \leq u_t(x_t)-C_1+dC_x
\end{align*}
so we just have to control $u_t(x_t)$.

We have $|m_t|=|\nu_t(x_t)|\leq C_m$, \emph{i.e}
\[
|tu_t(x_t)+(1-t)u_{\mathrm{ref}}(x_t)+j(x_t)| \leq C_m.
\]

Now we have:
\begin{itemize}
\item $t\geq t_0>0$,
\item $|j(x_t)|\leq C_2$ for some constant $C_2$ because $x_t\in b_1\rho + \mathfrak{a}^+$,
\item and $u_{\mathrm{ref}}(x_t)\leq \mathrm{sup}\{u_{\mathrm{ref}}(y)~;~y\in B(0,C_x)\}=:C_3$.
\end{itemize}
So 
\[
u_t(x_t)\leq \frac{C_m+C_2+C_3}{t_0}.
\]
Finally we have proved the uniform upper bound 
\[
(u_t-u_{\mathrm{ref}})(x)\leq C_4:= \frac{C_m+C_2+C_3}{t_0}-C_1+dC_x.
\qedhere
\]
\end{proof}

\section{Estimates on $|m_t|$}
\label{sec_m_t}

In this section, we essentially follow the work of Wang and Zhu to obtain 
estimates on $|m_t|$ and a uniform estimates on the (at least linear) growth of $\nu_t$.
A key step is proposition~\ref{propV}, which is a direct consequence of Proposition~\ref{degrees}
and will be used in the next sections also.

We consider the set 
\[
A_t := 
\{x\in \mathfrak{a}^+ ~;~ m_t\leq \nu_t(x) \leq m_t+1 \}\subset \mathfrak{a}^+.
\]
We will obtain upper and lower bound for the volume of $A_t$. 
The upper bound will depend on $m_t$.
The key is to obtain an upper bound that is small enough to give information.

Note that  the set $A_t$ is a bounded and convex set.
Indeed, since $m_t$ is the minimum of $\nu_t$, $A_t$ is a sublevel set of $\nu_t$ which is 
convex, so $A_t$ is convex. Furthermore, by the properness of $\nu_t$, $A_t$ is a  
bounded set.

\begin{lem}
\label{VolA}
There is an upper bound on the volume of $A_t$: 
\[
\mathrm{Vol}(A_t)\leq Ce^{m_t/2}
\]
where the constant $C>0$ does not depend on $t\geq t_0$. 
\end{lem}

\begin{proof}
Fritz John proved in \cite[Theorem III]{Joh48} 
that for any convex and bounded subset $A$ of $\mathbb{R}^r$, there exists an ellipsoid 
$E$ such that 
\[
\frac{1}{r} E\subset A \subset E
\]
where $(1/r) E$ is the dilation of $E$ of factor $1/r$ centered at the center 
of the ellipsoid $E$.

We can then find such an ellipsoid $E_t$ 
for $A_t$. 
Let $T$ be a linear transformation, of determinant one, such that $T(E)=B(y,\delta)$ is a ball. 
We will obtain an upper bound on $\delta$, thus getting an upper bound for the volume 
of $\mathrm{Vol}(A_t)$ because 
\[
\mathrm{Vol}(A_t)\leq \mathrm{Vol}(E)= \mathrm{Vol}(T(E))=C\delta^r.
\]

Let $\nu_t'$ be the function defined by $\nu_t'(x)=\nu_t(T^{-1}(x))$. 
We want to use a comparison principle on $B(y,\delta /r)$.
For that we first show that 
$\mathrm{MA}_{\mathbb{R}}(\nu_t')(x)\geq C e^{-m_t}$ on $T(A_t)$.
This is equivalent to showing that 
$\mathrm{MA}_{\mathbb{R}}(\nu_t)(x)\geq C e^{-m_t}$ on $A_t$.
First remark that since  the Hessian $\mathrm{Hess}_{\mathbb{R}}\nu_t$ of $\nu_t$ satisfies:
\[
\mathrm{Hess}_{\mathbb{R}}\nu_t = 
t\mathrm{Hess}_{\mathbb{R}}u_t + 
(1-t)\mathrm{Hess}_{\mathbb{R}}u_{\mathrm{ref}}+\mathrm{Hess}_{\mathbb{R}}j,
\]
we have 
\[
\mathrm{det}(\mathrm{Hess}_{\mathbb{R}}\nu_t)\geq \mathrm{det}(t\mathrm{Hess}_{\mathbb{R}}u_t),
\]
i.e.
\[
\mathrm{MA}_{\mathbb{R}}(\nu_t)(x)\geq t^r \mathrm{MA}_{\mathbb{R}}(u_t)(x).
\]
Using Proposition~\ref{eqnG} we deduce that 
\begin{align*}
\mathrm{MA}_{\mathbb{R}}(\nu_t)(x) & \geq 
t^r J(x)  e^{-w_t(x)} \prod_{\alpha \in  \Phi^+} \frac{1}{\left<\alpha , \nabla u_t(x)\right>^2}\\
& \geq t^r  e^{-\nu_t(x)} \prod_{\alpha \in  \Phi^+} \frac{1}{\left<\alpha , \nabla u_t(x)\right>^2}.
\end{align*}

We treat the factors separately: 
\begin{itemize}
\item We have $t\geq t_0>0$ for $t_0$ defined in Notation~\ref{notat0}.
\item By definition of $A_t$, we have $e^{-\nu_t(x)}\geq e^{-m_t-1}$ on $A_t$.
\item For any $x\in \mathfrak{a}$, we have $\nabla u_t(x) \in 2P$, so for any 
$\alpha \in  \Phi^+$, $\left<\alpha , \nabla u_t(x)\right>$ is bounded above independently of $t$.
This implies that 
\[
\prod_{\alpha \in  \Phi^+} \frac{1}{\left<\alpha , \nabla u_t(x)\right>^2} \geq c
\]
for some positive constant $c$.
\end{itemize}
In conclusion, we indeed have an inequality 
$\mathrm{MA}_{\mathbb{R}}(\nu_t)(x)\geq C e^{-m_t}$ on $A_t$, with $C$ a positive constant 
independent of $t\geq t_0$.

Now we use the comparison principle on $B(y,\delta /r)$ for real Monge-Ampère equations:  
let $g$ be the auxiliary function defined by
\[
g(x)=C^{1/r}e^{-m_t/r}\left(|x-y|^2-\frac{\delta^2}{r^2}\right)+m_t+1.
\]
Then we have
\begin{itemize} 
\item $g(x)=m_t+1\geq \nu_t'(x)$ for $x\in \partial B(y,\delta /r)$, and 
\item $\mathrm{MA}_{\mathbb{R}}(g)(x) = 
		C e^{-m_t} \leq \mathrm{MA}_{\mathbb{R}}(\nu_t')(x)$ 
		on $ B(y,\delta /r)$.
\end{itemize}
So the comparison principle gives that 
$\nu_t'(x)\leq g(x)$ on $B(y,\delta /r)$.
In particular, we have 
\begin{align*}
m_t & \leq \nu_t(T^{-1}(y)) \\
       & \leq \nu_t'(y) \\
       & \leq g(y)  \\
       & \leq C^{1/r}e^{-m_t/r}\left(-\frac{\delta^2}{r^2}\right)+m_t+1.
\end{align*}

We deduce from that the following upper bound for $\delta$:
\[
\delta\leq \sqrt{\frac{1}{C^{1/r}}}re^{m_t/2r}.
\]
Putting everything together, we obtain 
\[
\mathrm{Vol}(A_t)\leq \mathrm{Vol}(B(y,\delta) \leq C'e^{m_t/2}.
\qedhere
\]
\end{proof}

We turn now to a lower bound on $\mathrm{Vol}(A_t)$.   

\begin{lem}
\label{leqA}
There exists a constant $c>0$ independent of $t$ such that 
\[
\mathrm{Vol}(A_t)\geq c.
\]
\end{lem}

\begin{proof}
There exists a constant $b_2$ independent of $t$ such that $0<b_2<b_1$ and 
\[
A_t\subset b_2\rho + \mathfrak{a}^+.
\]
This is a corollary of Proposition~\ref{technical},
taking $b_2$ corresponding to $M=1$.

Then by Lemma~\ref{jderivatives} and Proposition~\ref{Potentials}, 
on $b_2\rho + \mathfrak{a}^+$, $|\nabla \nu_t|$ is bounded independently 
of $t$, say by $C$.
Then it is clear that the ball $B(x_t,1/C)$ is contained in $A_t$.
So $\mathrm{Vol}(A_t)\geq c$ for some $c>0$ independent of $t$.
\end{proof}

\begin{prop}
\label{propV}
The following integral is independent of $t$:
\[ 
\int_{\mathfrak{a}^+} e^{-\nu_t(x)}dx=\int_{2P^+}\prod_{\alpha\in \Phi^+}\left<\alpha , p\right>^2dp=:V.
\]
\end{prop}

\begin{proof}
Applying Proposition~\ref{degrees} with the ample line bundle $-K_X$, we have, for some 
constant $C$ depending only on $G$, and for any convex potential $u$ of a smooth 
$K\times K$-invariant positively curved hermitian metric on $-K_X$, 
\begin{align*}
\mathrm{deg}(-K_X) & = C \int_{\mathfrak{a}^+}\prod_{\alpha\in \Phi^+}\left<\alpha , \nabla u(x)\right>^2\mathrm{MA}_{\mathbb{R}}(u)(x)dx \\
	& = C \int_{2P^+}\prod_{\alpha\in \Phi^+}\left<\alpha , p\right>^2dp
\end{align*}

We apply this to the convex potential $u_t$, which by Proposition~\ref{eqnG} satisfies 
\[ 
e^{-\nu_t(x)}=\prod_{\alpha\in \Phi^+}\left<\alpha , \nabla u_t(x)\right>^2\mathrm{MA}_{\mathbb{R}}(u_t)(x)
\]
and obtain the statement, with $V=\mathrm{deg}(-K_X)/C$.
\end{proof}

We can now prove the main result of this subsection.
\begin{prop}
\label{linearGrowth}
\mbox{}
\begin{itemize}
\item There exists a constant $C$ independent of $t$, such that $|m_t|\leq C$.
\item There exist a constant $\kappa >0$ and a constant $C$, both independent of $t$, such that 
for $x\in \mathfrak{a}^+$,
\[
\nu_t(x) \geq \kappa |x-x_t|-C.
\]
\end{itemize}
\end{prop}

\begin{proof}
Following here Donaldson \cite{Don08} rather than Wang and Zhu, we write
\begin{align*}
\int_{\mathfrak{a}^+} e^{-\nu_t(x)}dx 
& = \int_{\mathfrak{a}^+} \int_{\nu_t(x)}^{+\infty}e^{-s}dsdx   \\
& = \int_{-\infty}^{+\infty} e^{-s} \int_{\mathfrak{a}^+}1_{\{\nu_t(x)\leq s\}}dxds \\
& = \int_{m_t}^{+\infty} e^{-s} \mathrm{Vol}(\{\nu_t\leq s\}) ds \\
& = e^{-m_t} \int_{0}^{+\infty} e^{-s} \mathrm{Vol}(\{\nu_t\leq m_t + s\}) ds 
\end{align*}

Now remark that $\{\nu_t\leq m_t + s\}\subset s\cdot A_t$ by convexity of $\nu_t$, 
where $s\cdot A_t$ is 
the $s$-dilation of $A_t$ with center $x_t$.
We deduce from that 
\[ 
\mathrm{Vol}(\{w_t\leq \nu_t + s\}) \leq s^r\mathrm{Vol}(A_t)\leq Cs^re^{m_t/2}.
\]
Applying this to the formula above we obtain 
\begin{align*}
\int_{\mathfrak{a}^+} e^{-\nu_t(x)}dx 
	& \leq e^{-m_t} C e^{m_t/2} \int_0^{+\infty}e^{-s}s^nds \\
	& \leq C'e^{-m_t/2}.
\end{align*}
The left hand side being constant, this inequality gives an upper bound on $m_t$.

For the lower bound, remark that
\begin{align*}
V= \int_{\mathfrak{a}^+} e^{-\nu_t(x)}dx 
& = e^{-m_t} \int_{0}^{+\infty} e^{-s} \mathrm{Vol}(\{w_t\leq \nu_t + s\}) ds \\
& \geq e^{-m_t} \int_{1}^{+\infty} e^{-s} \mathrm{Vol}(\{w_t\leq \nu_t + s\}) ds \\
& \geq e^{-m_t} \mathrm{Vol}(A_t) \int_{1}^{+\infty} e^{-s} ds. 
\end{align*}
By Lemma~\ref{leqA}, $\mathrm{Vol}(A_t) $ admits a lower bound $c$ independent of $t$, so 
\[
V\geq e^{-m_t} c \int_{1}^{+\infty} e^{-s} ds,
\]
and we deduce 
that 
\[ 
-m_t\leq \ln(V) - \ln\left(c \int_{1}^{+\infty} e^{-s} ds\right)
\]
so $-m_t$ is bounded  above by a constant independent of $t$.
Thus we have showed estimates on $|m_t|$.

Now for linear growth, the estimate on $|m_t|$ implies that we know both an upper bound $C_1$ and a lower bound $C_2$ independent of $t$ 
for the volume of $A_t$.
Since this set is convex, and contains a ball $B(x_t,\delta_0)$ of fixed radius $\delta_0$ independent of $t$  
by the proof of Lemma~\ref{leqA}, 
this implies that $A_t$ is included in a ball $A_t\subset B(x_t,\delta)$ where $\delta$ only depends on 
$C_1$ and $\delta_0$.   

By convexity of $\nu_t$, this implies that $\nu_t(x)\geq |x-x_t|/\delta +m_t$ outside of 
the ball, and we can 
extend this inequality to the whole of $\mathfrak{a}^+$ simply by subtracting 1:
\[
\nu_t(x)\geq |x-x_t|/\delta+m_t-1
\]
everywhere.
Using again the fact that $m_t$ is uniformly bounded we get the result with $\kappa = 1/\delta$.    
\end{proof}

\section{Obstruction, and upper bound on $R(X)$}
\label{sec_obstruction}

We will here explain how our condition appears as a necessary condition, then 
obtain an upper bound on the greatest Ricci lower bound. 
Everything relies on the following vanishing statement, which will also be 
of major importance in the next section.

\subsection{A vanishing integral}

\begin{prop}
\label{zeroInt}
Let $u$ be the convex potential of a smooth positive metric on $-K_X$. Define $\nu:=u+j$. 
Let $\xi$ be any vector in $\overline{\mathfrak{a}^+}$. Then 
\[
\int_{\mathfrak{a}^+} \frac{\partial \nu}{\partial \xi} e^{-\nu}dx=0.
\]
\end{prop}

Before we get to the proof, let us remark that the function considered is integrable. 
More generally, we can remark first that for any potential $u_0$, and any vector $\xi$, the
function $ \frac{\partial u_0}{\partial \xi} e^{-\nu}$ is integrable on $\mathfrak{a}^+$.
This is the case because $\nabla u_0 \in 2P$, and 
$e^{-\nu}\leq Ce^{-(v-4\rho)+C}$ (by Proposition~\ref{Potentials} and Lemma~\ref{Dominate}) is obviously 
integrable. 

Secondly, we have to show that the function 
$\frac{\partial j}{\partial \xi} e^{-\nu}= \frac{\partial j}{\partial \xi}J e^{-u}$
is integrable. 
Write
\begin{align*}
\frac{\partial j}{\partial \xi}(x)J(x) 
& = -2\sum_{\alpha\in \Phi^+} \left<\alpha , \xi\right>\coth\left<\alpha , x\right>\prod_{\beta\in \Phi^+}\sinh^2\left<\beta,x\right> \\
& = -2\sum_{\alpha \in \Phi^+} \left<\alpha , \xi\right> \cosh\left<\alpha , x\right>\sinh\left<\alpha , x\right>\prod_{\beta \neq \alpha}\sinh^2\left<\beta , x\right>.
\end{align*}
Then  by a computation similar to Lemma~\ref{Dominate}, we have 
\[
\left|e^{4\rho} \frac{\partial j}{\partial \xi}J\right|  \leq Ce^{4\rho},
\]
so again $\frac{\partial j}{\partial \xi} e^{-\nu}$ is integrable.

\begin{proof}
Choose a basis $(e_i)_{i=1..s}$ of the semisimple part $\mathfrak{a}_{ss}$ which generate 
$\mathfrak{a}^+\cap \mathfrak{a}_{ss}$ as a cone, 
and a basis $(f_j)_{j=1..r-s}$ of the toric part $\mathfrak{a}_t$. 
Consider the following sets, for $0\leq \epsilon<M$:
\[ 
Q(\epsilon, M):= 
\Big\{ \sum_i x_ie_i + \sum_j y_jf_j ~ ; ~ 
\forall i ~ \epsilon \leq x_i \leq M, ~ \forall j ~ -M\leq y_j\leq M \Big\}.
\]
Note that $\partial Q(\epsilon, M)=S_1(\epsilon, M) \cup S_2(\epsilon, M)$, where:
\begin{align*}
S_1(\epsilon,M)& := \Big\{ \sum_i x_ie_i + \sum_j y_jf_j \in Q(\epsilon, M)~;~ \exists i ~x_i=\epsilon \Big\}\\
S_2(\epsilon,M)& := \Big\{ \sum_i x_ie_i + \sum_j y_jf_j \in Q(\epsilon, M)~;~ \exists i ~x_i=M ~ \mathrm{or} ~ \exists j ~|y_j|=M \Big\}.
\end{align*}

Remark that $\frac{\partial \nu}{\partial \xi} e^{-\nu}=-\frac{\partial e^{-\nu}}{\partial \xi}$.
Then by the divergence formula applied to $e^{-\nu}$ we have for $\epsilon>0$,
\[
\int_{Q(\epsilon, M)} \frac{\partial \nu}{\partial \xi} e^{-\nu}dx =
\int_{S_1(\epsilon,M)\cup S_2(\epsilon,M)} e^{-\nu} \left<\xi,\mu\right> d\sigma 
\]
where $\mu$ is the exterior normal and $d\sigma$ is the surface area.

Write now $e^{-\nu}=e^{-u}J$. This is a continuous function on $\mathfrak{a}$, and it 
vanishes on the Weyl walls. 
Fixing $M$, we can thus let $\epsilon$ tend to 0, and we have that $e^{-\nu}$ 
tends uniformly to 0 on $S_1(\epsilon,M)$, so 
$\int_{S_1(\epsilon,M)} e^{-\nu} \left<\xi,\mu\right> d\sigma$ tends to 0.

We thus have 
\[
\int_{Q(0, M)} \frac{\partial \nu}{\partial \xi} e^{-\nu}dx =
\int_{S_2(0,M)} e^{-\nu} \left<\xi,\mu\right> d\sigma .
\]
Now as we have seen before, we have 
$e^{-\nu} \leq Ce^{-(v-4\rho)+C}$, so $e^{-\nu}(x)$ decreases exponentially as 
$|x|$ tends to infinity. Since the area of $S_2(0,M)$ grows polynomially, 
this ensures that 
$\int_{S_2(0,M)} e^{-\nu} \left<\xi,\mu\right> d\sigma$ tends to zero as $M$ tends 
to $\infty$.
This ends the proof.
\end{proof}

\subsection{Obstruction to the existence of a Kähler-Einstein metric}

Let us now explain the relation between convex potentials and the barycenter 
$\mathrm{bar}_{DH}(P^+)$.
Recall that $V$ denotes the volume of the polytope $P^+$ with respect to the 
Duistermaat-Heckman measure, and is equal, up to a multiplicative constant, 
to the volume of the anticanonical line bundle.  

\begin{prop}
\label{barDH}
Assume that $u$ is the convex potential of a smooth, positively curved hermitian 
metric on $-K_X$, then 
\[
\int_{\mathfrak{a}^+} \frac{\partial u}{\partial \xi} \prod_{\alpha\in\Phi^+}\left<\alpha , \nabla u\right>^2
\mathrm{MA}_{\mathbb{R}}(u)dx  = 
\left< \xi, \mathrm{bar}_{DH}(2P^+)\right> V.
\]
\end{prop}

\begin{proof}
We simply use the change of variable $p=\nabla u(x)$, just as in the proof of Proposition~\ref{degrees}.
It gives: 
\begin{align*}
\int_{\mathfrak{a}^+} \frac{\partial u}{\partial \xi} \prod_{\alpha\in\Phi^+}\left<\alpha ,\nabla u\right>^2
\mathrm{MA}_{\mathbb{R}}(u)dx 
&= \int_{2P^+} \left< p,\xi \right> \prod_{\alpha\in\Phi^+}\left<\alpha,p\right>^2dp \\
&= \left< \xi, \mathrm{bar}_{DH}(2P^+)\right> V.
\qedhere
\end{align*}
\end{proof}

We now prove an obstruction to the existence of Kähler-Einstein metrics on $X$. 
Recall that $\Xi$ denotes the relative interior of the cone generated by the 
positive roots $\Phi^+$.

\begin{prop}
\label{prop_obstruction}
Assume there exists a Kähler-Einstein metric on $X$. Then 
\[
\mathrm{bar}_{DH}(2P^+)  \in 4\rho + \Xi.
\]
\end{prop}

\begin{proof}
By Proposition~\ref{eqnG}, the Kähler-Einstein equation restricted to the open orbit reads:
\[
\mathrm{MA}_{\mathbb{R}}(u)\prod_{\alpha\in\Phi^+}\left<\alpha,\nabla u\right>^2=e^{-u}J.
\]
Suppose that there exists a solution $u$.
Applying Proposition~\ref{zeroInt} to $u$ gives 
\[
\int_{\mathfrak{a}^+} \frac{\partial \nu}{\partial \xi} e^{-\nu}dx=0,
\]
so by definition of $\nu=u+j$, 
\[
\int_{\mathfrak{a}^+} \frac{\partial u}{\partial \xi} e^{-\nu}dx + 
\int_{\mathfrak{a}^+} \frac{\partial j}{\partial \xi} e^{-\nu}dx = 0.
\]

Since $u$ is solution to the Kähler-Einstein equation, we have 
\[
e^{-\nu}=e^{-u}J=\mathrm{MA}_{\mathbb{R}}(u)\prod_{\alpha\in\Phi^+}\left<\alpha ,\rho\right>^2.
\]
In particular, by Proposition~\ref{propV},
\[
\int_{\mathfrak{a}^+} e^{-\nu}dx = V
\]
is constant, and by Proposition~\ref{barDH}, 
\[
\int_{\mathfrak{a}^+} \frac{\partial u}{\partial \xi} e^{-\nu}dx
= \left< \xi, \mathrm{bar}_{DH}(2P^+)\right> V.
\]

By Lemma~\ref{j4rho}, we have for $\xi \in \mathfrak{a}^+$, 
\[
\int_{\mathfrak{a}^+} \frac{\partial j}{\partial \xi} e^{-\nu}dx < -4\left<\rho,\xi\right>V.
\]
Combining these facts, we have:
\begin{align*}
\left< \xi, \mathrm{bar}_{DH}(2P^+)\right> V -4\left<\rho,\xi\right>V &> 
\int_{\mathfrak{a}^+} \frac{\partial u}{\partial \xi} e^{-\nu}dx+ \int_{\mathfrak{a}^+} \frac{\partial j}{\partial \xi} e^{-\nu}dx \\
&>0.
\end{align*}
Dividing by $V$ we obtain that for any $\xi \in \mathfrak{a}^+$, 
\[
\left< \xi, \mathrm{bar}_{DH}(2P^+)-4\rho \right> >0.
\]
The large inequality would mean $\mathrm{bar}_{DH}(2P^+)-4\rho \in (\mathfrak{a}^+)^{\vee}$, 
and here the strict inequality means that $\mathrm{bar}_{DH}(2P^+)-4\rho$ 
is in the relative interior of $(\mathfrak{a}^+)^{\vee}$, which is precisely $\Xi$.
\end{proof}

\begin{exa}
This obstruction allows to show that the manifold $X_2$ does not admit any 
Kähler-Einstein metric. Recall that this is a compactification of 
$\mathrm{Sp}_4(\mathbb{C})$ so the corresponding root system is $B_2$.
Denote by $\alpha_1$ and $\alpha_2$ the simple roots, with $\alpha_2$ the 
long root. Then the other positive roots are $\alpha_1+\alpha_2$ and $2\alpha_1+\alpha_2$,
so $2\rho=4\alpha_1 + 3\alpha_2$.

Choosing a realization of $B_2$ in the euclidean plane with 
$\alpha_1=(1,0)$ and $\alpha_2=(-1,1)$, we compute that,
for $p=x\alpha_1 + y\alpha_2$,
\[
\prod_{\alpha\in \Phi^+} \left<\alpha , p\right>^2 = x^2y^2(x-y)^2(-x+2y)^2.
\]
We can then compute 
\[
\mathrm{bar}_{DH}(P^+)=
\frac{278037566905}{66955221696}\alpha_1 +\frac{3043253830}{1046175339}\alpha_2.
\]
In particular the coordinate in $\alpha_2$ is strictly smaller than $3$ so there 
exists no Kähler-Einstein metric on $X_2$.
Figure~\ref{BarX2} gives a representation of the polytope $P^+$, the barycenter is represented
by a cross and the cone delimited by the dashed lines is $2\rho + \Xi$. 

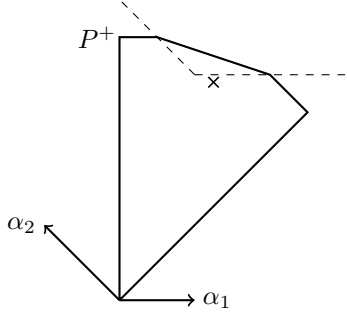
\begin{figure}
\centering
\begin{tikzpicture}[scale=1]
\draw[thick] (0,0) -- (5/2,5/2) -- (2,3) -- (1/2,7/2) -- (0,7/2) -- (0,0);
\draw (-0.3,7/2) node {$P^+$};
\draw[->, thick] (0,0) -- (1,0);
\draw (1.3,0) node {$\alpha_1$};
\draw[->, thick] (0,0) -- (-1,1);
\draw (-1-0.3,1) node {$\alpha_2$};
\draw[dashed] (1,3) -- (3,3);
\draw[dashed] (1,3) -- (0,4);
\draw[semithick] (1.25-0.07,2.9-0.07) -- (1.25+0.07,2.9+0.07);
\draw[semithick] (1.25+0.07,2.9-0.07) -- (1.25-0.07,2.9+0.07);
\end{tikzpicture}
\caption{Barycenter for $X_2$}
\label{BarX2}
\end{figure}

\end{exa}

\subsection{Upper bound on $R(X)$}

\begin{prop}
\label{prop_upper}
Assume that $X$ admits no Kähler-Einstein metrics, then the greatest Ricci lower bound of $X$ is lower than 
or equal to the supremum of all $t<1$ such that 
\[
4\rho + \frac{t}{t-1}( \mathrm{bar}_{DH}(2P^+)-4\rho)\in -\Xi+2P^+.
\]
\end{prop}

\begin{proof}
Consider the equation at time $t$:
\[
\mathrm{MA}_{\mathbb{R}}(u_t)\prod_{\alpha\in\Phi^+}\left<\alpha,\nabla u_t\right>^2=e^{-\nu_t}.
\]
Apply Proposition~\ref{zeroInt} to $w_t$. This gives for any $\xi \in \overline{\mathfrak{a}^+}$,
\[
\int_{\mathfrak{a}^+} \frac{\partial \nu_t}{\partial \xi} e^{-\nu_t}dx=0.
\]
This is equivalent to 
\[
t\int_{\mathfrak{a}^+} \frac{\partial u_t}{\partial \xi}e^{-\nu_t} + 
(1-t)\int_{\mathfrak{a}^+} \frac{\partial u_{\mathrm{ref}}}{\partial \xi}e^{-\nu_t}
+\int_{\mathfrak{a}^+} \frac{\partial j}{\partial \xi}e^{-\nu_t}=0.
\]

Without loss of generality we can assume $t<1$ and divide by $(t-1)V$ to get 
\[ 
\frac{t}{t-1} \int_{\mathfrak{a}^+} \frac{\partial u_t}{\partial \xi}\frac{e^{-\nu_t}}{V} + 
\frac{1}{t-1}\int_{\mathfrak{a}^+} \frac{\partial j}{\partial \xi}\frac{e^{-\nu_t}}{V}
= \int_{\mathfrak{a}^+} \frac{\partial u_{\mathrm{ref}}}{\partial \xi}\frac{e^{-\nu_t}}{V}.
\]
If $v$ is the support function of $2P$, we have for any $x\in \mathfrak{a}^+$, 
$\frac{\partial u_{\mathrm{ref}}}{\partial \xi}(x) \leq v(\xi)$, so 
\[
\int_{\mathfrak{a}^+} \frac{\partial u_{\mathrm{ref}}}{\partial \xi}\frac{e^{-\nu_t}}{V} \leq v(\xi).
\]
On the other hand, we can use here also Lemma~\ref{j4rho}  
to get, for $\xi \in \mathfrak{a}^+$,  
\[
\frac{1}{t-1}\int_{\mathfrak{a}^+} \frac{\partial j}{\partial \xi}\frac{e^{-\nu_t}}{V} > \frac{1}{t-1}(-4\left<\rho,\xi\right>).
\]
We thus have, using Proposition~\ref{barDH},
\begin{align*}
v(\xi) & > 
\frac{t}{t-1} \int_{\mathfrak{a}^+} \frac{\partial u_t}{\partial \xi}\frac{e^{-\nu_t}}{V} +\frac{1}{t-1}\int_{\mathfrak{a}^+} \frac{\partial j}{\partial \xi}\frac{e^{-\nu_t}}{V}\\
& > \frac{t}{t-1}\left< \xi, \mathrm{bar}_{DH}(2P^+) \right> - \frac{1}{t-1}4\left<\rho,\xi\right>\\
& > \left<\xi, 4\rho + \frac{t}{t-1}( \mathrm{bar}_{DH}(2P^+)-4\rho)\right>
\end{align*}

The fact that this is true for all $\xi \in \mathfrak{a}^+$ means, since $v$ is the support 
function of $2P$, that
\[
4\rho + \frac{t}{t-1}( \mathrm{bar}_{DH}(2P^+)-4\rho)\in -\Xi + 2P^+.
\qedhere
\]
\end{proof}

\section{Absence of estimates on $|x_t|$}
\label{sec_absence}

\subsection{Consequence}

We will assume now that there are no Kähler-Einstein metrics on $X$. We will denote by 
$t_{\infty}:=R(X)$ 
the greatest Ricci lower bound. 

Our assumption implies that $|x_t|$ is unbounded as $t$ tends to $t_{\infty}$.
Indeed if it was not the case, then we would have estimates on $|x_t|$ and so by 
Proposition~\ref{lem_est_m_x} and Lemma~\ref{linearGrowth}, 
there would be a solution at time $t_{\infty}$ and by openness for times 
greater than $t_{\infty}$ if $t_{\infty}$. This is a contradiction.

We can find a sequence $t_i$ 
such that $t_i\rightarrow t_{\infty}$ and $|x_{t_i}|\rightarrow \infty$. 
Define $\xi_t \in \mathfrak{a}^+$ as the direction given by the minimum $\xi_t := x_t/|x_t|$.
Up to taking a subsequence, we can also assume that $\xi_t$ admits a limit $\xi_{\infty}\in \overline{\mathfrak{a}^+}$ 
as $t_i\rightarrow t_{\infty}$.

We will consider an integral equality involving $\nu_t$ and consider the limit 
as $t_i\rightarrow t_{\infty}$.
The integral equality follows from Proposition~\ref{zeroInt} applied to 
$w_t$:
\[
\int_{\mathfrak{a}^+}\frac{\partial \nu_t}{\partial \xi}e^{-\nu_t}dx=0.
\]
Recall that $w_t=tu_t+(1-t)u_{\mathrm{ref}}$ by definition, so 
\[
\nu_t=tu_t+(1-t)u_{\mathrm{ref}}+j=t(u_t+j)+(1-t)(u_{\mathrm{ref}}+j).
\]
The vanishing integral thus gives 
\[ 
t \int_{\mathfrak{a}^+}\frac{\partial u_t}{\partial \xi_t}e^{-\nu_t}dx 
+ (1-t) \int_{\mathfrak{a}^+}\frac{\partial u_{\mathrm{ref}}+j}{\partial \xi_t}e^{-\nu_t}dx 
+ \int_{\mathfrak{a}^+}\frac{\partial j}{\partial \xi_t}e^{-\nu_t}dx =0,
\]
which can also be written 
\[
t \int_{\mathfrak{a}^+}\frac{\partial u_t+j}{\partial \xi_t}e^{-\nu_t}dx= (t-1) \int_{\mathfrak{a}^+}\frac{\partial u_{\mathrm{ref}}+j}{\partial \xi_t}e^{-\nu_t}dx.
\]

We will compute the limit of each of these terms as $t_i\rightarrow t_{\infty}$, and obtain the following result.

\begin{prop}
\label{limiteqn}
We have 
\[
t_{\infty} \left<\mathrm{bar}_{DH}(2P^+)-4 \rho , \xi_{\infty} \right>
= (t_{\infty}-1) (v( \xi_{\infty})-\left<4\rho , \xi_{\infty} \right>).
\]
\end{prop}

\subsection{Proof of Proposition~\ref{limiteqn}}  

Let us first consider 
\[
\int_{\mathfrak{a}^+}\frac{\partial u_t}{\partial \xi_t}e^{-\nu_t}dx.
\]
Since 
\[
e^{-\nu_t} = 
\prod_{\alpha\in \Phi^+} \left<\alpha,\nabla u_t\right>^2 \mathrm{MA}_{\mathbb{R}}(u_t),
\]
Proposition~\ref{barDH} implies 
\[
\int_{\mathfrak{a}^+}\frac{\partial u_t}{\partial \xi_t}e^{-\nu_t}dx =  \left< \xi_t , \mathrm{bar}_{DH}(2P^+) \right>V.
\]
In particular, the limit as $t_i \rightarrow t_{\infty}$ is 
\[ 
\left< \xi_{\infty} , \mathrm{bar}_{DH}(2P^+) \right>V.
\]

For the other terms we need more work to compute the limits.
For simplicity, we will often omit the indices $i$ in $t_i$.
We will prove the two following propositions.

\begin{prop}
\label{limj}
We have 
\[
\lim_{t_i\rightarrow t_{\infty}} 
		\int_{\mathfrak{a}^+} \frac{\partial j}{\partial \xi_t} e^{-\nu_t} = 
			-\left<4\rho,\xi_{\infty}\right> V.
\]
\end{prop}

\begin{prop}
\label{limu}
We have 
\[
\lim_{t_i\rightarrow t_{\infty}} \int_{\mathfrak{a}^+} 
	\frac{\partial u_{\mathrm{ref}}}{\partial \xi_t} e^{-\nu_t} 
= v(\xi_{\infty}) V.
\]
\end{prop}

Let us first find a domain $D(\epsilon)$ of the form 
$B(x_t,\delta)\cap (b\rho+\mathfrak{a}^+)$
where $e^{-\nu_t}dx$ puts all the mass up to $2\epsilon >0$.
When we write $B(x_t,\delta)$, we in general mean $B(x_t,\delta)\cap \mathfrak{a}^+$.

\begin{lem}
\label{Ball}
For any $\epsilon>0$, there exists a constant $\delta=\delta(\epsilon)$ independent of $t$ such that 
\[
\int_{\mathfrak{a}^+ \setminus B(x_t,\delta)}e^{-\nu_t}dx < \epsilon
\]
and 
\[
e^{-\kappa \delta+C}\sigma_r\delta^{r-1}< \epsilon,
\]
where $\sigma_r$ is the area of a sphere of radius 1 in $\mathbb{R}^r$.
\end{lem}

\begin{proof}
Recall from Proposition~\ref{linearGrowth} that $\nu_t(x)\geq \kappa|x-x_t|-C$, for some $\kappa>0, C$ independent of $t$.  
Observe that the function $e^{-\kappa|x-x_t|+C}$ is well defined on $\mathfrak{a}$, positive 
and integrable.
So there exists a $\delta>0$ such that 
\[
\int_{\mathfrak{a} \setminus B(x_t,\delta)}e^{-\kappa|x-x_t|+C}dx< \epsilon.
\]
But then we also have 
\[
\int_{\mathfrak{a}^+ \setminus B(x_t,\delta)}e^{-\nu_t}dx \leq 
\int_{\mathfrak{a}^+ \setminus B(x_t,\delta)}e^{-\kappa|x-x_t|+C}< \epsilon.
\]

Of course, since $e^{-\kappa y+C}$ decreases exponentially with respect to $y$,
we can increase $\delta$ so as to 
have the second condition:
\[
e^{-\kappa \delta+C}\sigma_r\delta^{r-1}< \epsilon.
\qedhere
\]
\end{proof}

\begin{lem}
\label{Domain}
For any $\epsilon >0$, let $\delta=\delta(\epsilon)$ be given by Lemma~\ref{Ball}. There exists a 
constant $b=b(\epsilon)>0$ such that
if we denote by $D=D(\epsilon)$ the domain $B(x_t,\delta)\cap (b\rho+\mathfrak{a}^+)$, 
we have 
\[
\int_{B(x_t, \delta) \setminus D}e^{-\nu_t}dx<\epsilon,
\]
and
\[
\int_{\partial D}e^{-\nu_t}d\sigma < \epsilon, 
\]
where $d\sigma$ is the area measure of $\partial D$, which is piecewise smooth.
\end{lem}

\begin{proof}
Here we want to use Proposition~\ref{technical}.
Now that we know that $m_t$ is uniformly bounded, we can choose $M$ and the corresponding 
$b$ so that: 
\[ 
\forall x \in \mathfrak{a}^+ \setminus (b\rho + \mathfrak{a}^+), \quad
e^{-\nu_t(x)}  < \max\left( \frac{\epsilon}{\sigma_r\delta^{r-1}}, \frac{\epsilon}{\delta^r\omega_r} \right)
\]
where $\omega_r$ is the volume of the ball of radius 1 in $\mathbb{R}^r$.

Let us prove that $e^{-\nu_t}< \epsilon /(\sigma_r\delta^{r-1})$ on $\partial D$.
A point $x\in \partial D$ is either on the sphere of radius $\delta$ centered at $x_t$, 
or on $\partial (b\rho+\mathfrak{a}^+)$.
In the first case, we have, by Proposition~\ref{linearGrowth},
\begin{align*}
e^{-\nu_t(x)} & \leq e^{-\kappa |x-x_t|+C} \\
& \leq e^{-\kappa \delta+C} \\
& <\frac{\epsilon}{\sigma_r\delta^{r-1}}
\end{align*}
by the second consequence of Lemma~\ref{Ball}.
In the second case, $x\in \partial (b\rho + \mathfrak{a}^+)$, so by the choice 
of $b$ above, using the first term in the maximum, we have 
\[
e^{-\nu_t(x)}<\frac{\epsilon}{\sigma_r\delta^{r-1}}.
\]

Obviously the volume of $\partial D$ is $\leq \sigma_r\delta^{r-1}$,
so 
\begin{align*}
\int_{\partial D}e^{-\nu_t}d\sigma
& <\int_{\partial D} \frac{\epsilon}{\sigma_r\delta^{r-1}}d\sigma \\
&< \epsilon.
\end{align*}

For the other conclusion, we use the fact that 
$e^{-\nu_t(x)} < \epsilon /(\delta^r\omega_r)$ on 
$B(x_t, \delta) \setminus D \subset \mathfrak{a}^+ \setminus (b\rho + \mathfrak{a}^+)$, 
which implies that
\begin{align*}
\int_{B(x_t, \delta) \setminus D}e^{-\nu_t}dx & < \int_{B(x_t, \delta) \setminus D}\frac{\epsilon}{\delta^r\omega_r}dx < \epsilon
\end{align*}
using the fact that the volume of $B(x_t, \delta) \setminus D$ is $\leq \delta^r\omega_r$.
\end{proof}

\begin{lem}
\label{djsurDc}
Let $\epsilon >0$ and $D=D(\epsilon)$ be the domain given by Lemma~\ref{Domain}. We have 
\[
\left| \int_{\mathfrak{a}^+ \setminus D} \frac{\partial j}{\partial \xi_t} e^{-\nu_t}  \right| 
< (2d+1)\epsilon.
\]
\end{lem}

\begin{proof}
Let us write:
\begin{align*}
\int_{\mathfrak{a}^+\setminus D}\frac{\partial j}{\partial \xi_t}e^{-\nu_t}dx
& =\int_{\mathfrak{a}^+\setminus D}\frac{\partial (\nu_t-w_t)}{\partial \xi_t}e^{-\nu_t}dx\\
& = \int_{\mathfrak{a}^+}\frac{\partial \nu_t}{\partial \xi_t}e^{-\nu_t}dx
	-\int_{D}\frac{\partial \nu_t}{\partial \xi_t}e^{-\nu_t}dx
	-\int_{\mathfrak{a}^+\setminus D}\frac{\partial w_t}{\partial \xi_t}e^{-\nu_t}dx
\end{align*}
The first of these three integrals is zero by Proposition~\ref{zeroInt}.

For the third term, we have 
\begin{align*}
\int_{\mathfrak{a}^+\setminus D}\left|\frac{\partial w_t}{\partial \xi_t}\right|e^{-\nu_t}dx 
& \leq d \int_{\mathfrak{a}^+\setminus D}e^{-\nu_t}dx\\
& < 2d\epsilon
\end{align*}
by Lemma~\ref{Domain}.

It remains to deal with the second integral. We apply the divergence theorem on $D$ 
to the function $e^{-\nu_t}$:
\[
\int_{D}\frac{\partial \nu_t}{\partial \xi_t}e^{-\nu_t}dx = 
-\int_{D}\frac{\partial e^{-\nu_t}}{\partial \xi_t}dx
= \int_{\partial D} e^{-\nu_t} \left< \xi_t, n \right> d\sigma 
\]
where $n$ is the exterior normal to $D$ and $d\sigma$ the area measure.

We conclude using Lemma~\ref{Domain} that 
\[
\left| \int_{D}\frac{\partial \nu_t}{\partial \xi_t}e^{-\nu_t}dx\right| 
\leq \int_{\partial D} e^{-\nu_t} d\sigma
<\epsilon.
\]

Putting everything together, we see that 
\[
\left| \int_{\mathfrak{a}^+\setminus D}\frac{\partial j}{\partial \xi_t}e^{-\nu_t}dx \right| 
< (2d+1)\epsilon 
\qedhere
\]
\end{proof}

\begin{proof}[Proof of Proposition~\ref{limj}]
Let $\epsilon >0$. Set $\theta = \epsilon /(3(2d+1+8|\rho|))$ and let $D:=D(\theta)$.
Write
\begin{align*}
\left|  \int_{\mathfrak{a}^+} \frac{\partial j}{\partial \xi_t} e^{-\nu_t} +\left<4\rho,\xi_{\infty}\right>V   \right| 
&  \leq \left| \int_{\mathfrak{a}^+ \setminus D} \frac{\partial j}{\partial \xi_t} e^{-\nu_t}   \right| 
          + \left|  \int_{D} \frac{\partial j}{\partial \xi_t} e^{-\nu_t} +\left<4\rho,\xi_{\infty}\right>V    \right|  \\
& \leq (2d+1)\theta + \left|  \int_{D} \frac{\partial j}{\partial \xi_t} e^{-\nu_t} +\left<4\rho,\xi_{\infty}\right>V    \right|
\end{align*}
by Lemma~\ref{djsurDc}.

Then we have 
\begin{align*}
\left|  \int_{D} \frac{\partial j}{\partial \xi_t} e^{-\nu_t} +\left<4\rho,\xi_{\infty}\right>V    \right| 
& \leq  \left|  \int_{D} \frac{\partial j}{\partial \xi_t} e^{-\nu_t} +\left<4\rho,\xi_t\right>V    \right| 
		+\left|  \left<4\rho,\xi_{\infty}-\xi_t\right>V  \right|
\end{align*}
The second term tends to zero so there exists an $i_0$ such that for all $i\geq i_0$, 
\[
\left|  \left<4\rho,\xi_{\infty}-\xi_t\right>V  \right| \leq \frac{\epsilon}{3}.
\]

We now deal with the second term:
\begin{align*}
\left|  \int_{D} \frac{\partial j}{\partial \xi_t} e^{-\nu_t} +\left<4\rho,\xi_t\right>V    \right| & 
\leq \left| \int_{D} \left(\frac{\partial j}{\partial \xi_t} + \left<4\rho,\xi_t\right>\right)e^{-\nu_t} \right| 
+ \left| \int_{\mathfrak{a}^+ \setminus D} \left<4\rho,\xi_t\right>e^{-\nu_t} \right| \\
& \leq \left| \int_{D} \left(\frac{\partial j}{\partial \xi_t} + \left<4\rho,\xi_t\right>\right)e^{-\nu_t} \right| 
+ \left|  \left<4\rho,\xi_t\right> \int_{\mathfrak{a}^+ \setminus D}e^{-\nu_t} \right| \\
& \leq \left| \int_{D} \left(\frac{\partial j}{\partial \xi_t} + \left<4\rho,\xi_t\right>\right)e^{-\nu_t} \right| 
+ |4\rho| \left|  \int_{\mathfrak{a}^+ \setminus D}e^{-\nu_t} \right| \\
\intertext{and by construction of $D$ we deduce:}
& \leq \left| \int_{D} \left(\frac{\partial j}{\partial \xi_t} + \left<4\rho,\xi_t\right>\right)e^{-\nu_t} \right| 
+ |4\rho|\cdot 2\theta
\end{align*}

 We consider now the quantity
\[ 
\frac{\partial j}{\partial \xi_t}(x) + \left<4\rho,\xi_t\right>
\]
for $x\in D$.
It is negative by Lemma~\ref{j4rho} since $\xi_t \in \overline{\mathfrak{a}^+}$.

Recall that $D(\theta) \subset b(\theta)\rho + \mathfrak{a}^+$ for some $b(\theta) >0$, 
and more precisely that 
$D(\theta) = B(x_t, \delta(\theta))\cap  (b(\theta)\rho + \mathfrak{a}^+)$.
Choose $b_0>0$ such that 
$B(b_0\rho,   \delta(\theta))\subset  b(\theta)\rho + \mathfrak{a}^+$.

We can write 
\begin{align*}
\xi_t & = \frac{x_t}{|x_t|} \\
& = \frac{x_t-b_0\rho}{|x_t|} + \frac{b_0\rho}{|x_t|}\\
&= \frac{ |x_t-b_0\rho| }{|x_t|} \xi_1 +\frac{|b_0\rho|}{|x_t|}\xi_2
\end{align*}
where $\xi_1 = (x_t-b_0\rho) / |x_t-b_0\rho|$ and 
$\xi_2= b_0\rho / |b_0\rho|$.

It gives 
\[
\frac{\partial}{\partial \xi_t} = 
	\frac{ |x_t-b_0\rho| }{|x_t|} \frac{\partial}{\partial \xi_1} 
		+\frac{|b_0\rho|}{|x_t|}\frac{\partial}{\partial \xi_2}.
\]

Let $x= x_t+y \in D$, we consider the restriction of $j$ to the line starting 
from $b_0\rho +y$ and of direction $\xi_1$, which contains $x$. 
By convexity, we have 
\[
\frac{\partial j}{\partial \xi_1}(x) \geq \frac{j(x)-j(y+b_0\rho)}{|x_t-b_0\rho|}.
\]

Recall from Lemma~\ref{Dominate}
that 
\[
j(x)\geq -2\sum_{\alpha \in \Phi^+}\left<\alpha , x\right>+C= -\left<4\rho,x\right>+C
\]
on $\mathfrak{a}^+$, for some constant $C$. 

Applying this gives 
\[
\frac{\partial j}{\partial \xi_1}(x) 
	\geq \frac{-\left<4\rho,x\right>+C-j(y+b_0\rho)}{|x_t-b_0\rho|}.
\]

Now going back to $\frac{\partial j}{\partial \xi_t}(x)$, we have 
\begin{align*}
\frac{\partial j}{\partial \xi_t}(x) 
& = \frac{ |x_t-b_0\rho| }{|x_t|} \frac{\partial j}{\partial \xi_1}(x) 
		+\frac{|b_0\rho|}{|x_t|}\frac{\partial j}{\partial \xi_2}(x) \\
& \geq \frac{ |x_t-b_0\rho| }{|x_t|}\frac{-\left<4\rho,x\right>
		+C-j(y+b_0\rho)}{|x_t-b_0\rho|} 
		+\frac{|b_0\rho|}{|x_t|}\frac{\partial j}{\partial \xi_2}(x) \\
 & \geq \frac{-\left<4\rho, x_t+y\right>+C-j(y+b_0\rho)}{|x_t|} 
		+\frac{|b_0\rho|}{|x_t|}\frac{\partial j}{\partial \xi_2}(x)
\end{align*}

so 

\[
0\geq \frac{\partial j}{\partial \xi_t}(x)+ \left<4\rho,\frac{x_t}{|x_t|}\right>  
	\geq \frac{-\left<4\rho,y\right>+C-j(y+b_0\rho) 
		+  |b_0\rho| \frac{\partial j}{\partial \xi_2}(x)    }{|x_t|}
\]

Now $y\in B(b_0\rho,   \delta(\theta))$ is bounded, $j$ is bounded on 
$b(\theta)\rho+\mathfrak{a}^+$ by Lemma~\ref{Dominate}, 
and the derivatives of $j$ are bounded 
on $b(\theta)\rho+\mathfrak{a}^+$ by Lemma~\ref{jderivatives}, so there is a (negative) constant $C'$ such that 
\[
0\geq \frac{\partial j}{\partial \xi_t}(x)+ \left<4\rho,\frac{x_t}{|x_t|}\right>  
\geq \frac{C'}{|x_t|}.
\]

Applying this to the sequence $t_i$, we find that there exists a $i_1$ such that for 
$i>i_1$, and unifomrly for $x\in D$,
\[
\left| \frac{\partial j}{\partial \xi_t}(x)+ \left<4\rho,\xi_t \right> \right| 
	\leq \frac{\epsilon}{3V}.
\]

Then for $i>i_1$, 
\begin{align*}
\left| \int_{D} \left(\frac{\partial j}{\partial \xi_t} 
	+ \left<4\rho,\xi_t\right>\right)e^{-\nu_t} \right| 
& \leq  \int_{D} \left|\frac{\partial j}{\partial \xi_t} + \left<4\rho,\xi_t\right>\right|e^{-\nu_t} \\
& \leq \frac{\epsilon}{3V} \int_{D}e^{-\nu_t} \\
& \leq \frac{\epsilon}{3}
\end{align*}

Gathering everything gives, for $i>\mathrm{max}(i_0,i_1)$,
\[
\left|  \int_{\mathfrak{a}^+} \frac{\partial j}{\partial \xi_t} e^{-\nu_t} 
	+\left<4\rho,\xi_{\infty}\right>V   \right|
\leq (2d+1)\theta + 8|\rho|\theta + \frac{\epsilon}{3} + \frac{\epsilon}{3} 
=\epsilon.
\qedhere
\]
\end{proof}

\begin{proof}[Proof of Proposition~\ref{limu}]
Let $\epsilon>0$. 
Set $\theta := \epsilon / 6d$ and let $\delta=\delta(\theta)$.
First, by Lemma~\ref{Ball}, we have 
\[
\left|\int_{\mathfrak{a}^+\setminus B(x_t,\delta)}
\frac{\partial u_{\mathrm{ref}}}{\partial \xi_t} e^{-\nu_t} \right| < d\theta.
\]

On $B(x_t,\delta)$ , we always have $\frac{\partial u_{\mathrm{ref}}}{\partial \xi_t}\leq v(\xi_t)$.
Now consider the ray starting from $x-x_t$ and going to $x$. By convexity, we have 
\begin{align*}
\frac{\partial u_{\mathrm{ref}}}{\partial \xi_t}(x)  & \geq \frac{u_{\mathrm{ref}}(x)-u_{\mathrm{ref}}(x-x_t)}{|x_t|}\\
& \geq \frac{v(x)+C}{|x_t|}\\
\intertext{for some constant $C$ independent of $x$ in $B(x_t,\delta)$, by 
Proposition~\ref{Potentials} and because $u_{\mathrm{ref}}$ is bounded on $B(0,\delta)$.
Then we can write }
\frac{\partial u_{\mathrm{ref}}}{\partial \xi_t}(x) & \geq v(\xi_t+\frac{x-x_t}{|x_t|})+\frac{C}{|x_t|} \\
& \geq v(\xi_t)+\frac{C'}{|x_t|}
\end{align*}
The last step holds because $v$ is Lipschitz.

For $i>i_0$ for some $i_0$, we thus have, for $x\in B(x_t,\delta)$,
\[
\left|\frac{\partial u_{\mathrm{ref}}}{\partial \xi_t}(x) - 
v(\xi_t)\right|<\frac{\epsilon}{3}  V.
\]
Integrating on the ball gives 
\[
\left|\int_{B(x_t,\delta)}\frac{\partial u_{\mathrm{ref}}}{\partial \xi_t} e^{-\nu_t}
	-\int_{B(x_t,\delta)} v(\xi_t)e^{-\nu_t}\right|<\frac{\epsilon}{3}.
\]

Applying Lemma~\ref{Ball} again gives 
\[
\left|\int_{B(x_t,\delta)} v(\xi_t)e^{-\nu_t}-\int_{\mathfrak{a}^+} v(\xi_t)e^{-\nu_t}\right|< d \theta,
\]
with
\[
\int_{\mathfrak{a}^+} v(\xi_t)e^{-\nu_t}=v(\xi_t)V.
\]

Finally, since $\xi_t$ converges to $\xi_{\infty}$, there exists $i_1$ such that for $i>i_1$, 
\[
|v(\xi_t)V-v(\xi_{\infty})V| < \frac{\epsilon}{3}.
\]

We have proved that for $i>i_0,i_1$, we have 
\[
\left|\int_{\mathfrak{a}^+} 	\frac{\partial u_{\mathrm{ref}}}{\partial \xi_t} e^{-\nu_t} 
   -v(\xi_{\infty})V\right|<2\frac{\epsilon}{3} + 2d \theta = \epsilon
\qedhere
\]
\end{proof}

This concludes the proof of Proposition~\ref{limiteqn}, by considering the limit 
at $t_{\infty}$ and dividing by $V$.

\subsection{Sufficient condition}

We can now prove that our condition is sufficient for the existence of a Kähler-Einstein 
metric and thus conclude the proof of Theorem~A.

\begin{thm}
\label{thm_sufficient}
If $\mathrm{bar}_{DH}(2P^+) \in 4\rho +\Xi$, then $X$ admits a 
Kähler-Einstein metric.
\end{thm}

\begin{proof}
Assume that $X$ admits no Kähler-Einstein metric. 
Then Proposition~\ref{limiteqn} gives 
\[
t_{\infty} \left<\mathrm{bar}_{DH}(2P^+)-4 \rho ,\xi_{\infty}\right>
= (t_{\infty}-1)(v(\xi_{\infty})-\left<4\rho,\xi_{\infty}\right>).
\]
In particular, since $v$ is the support function of $2P$ and $2\rho \in \mathrm{Int}(P)$, and 
$0<t_{\infty} \leq 1$, we have 
\[
\left<\mathrm{bar}_{DH}(2P^+)-4 \rho ,\xi_{\infty}\right> \leq 0.
\]

Assume that $\mathrm{bar}_{DH}(2P^+) \in 4\rho +\Xi$.
Then by the definition of $\Xi$, the only possibility is that 
$\xi_{\infty} \in \mathfrak{a}_{t}$ and $\left<\mathrm{bar}_{DH}(2P^+)-4 \rho ,\xi_{\infty}\right>=0$.

To prove that this is impossible, consider the vanishing 
\[
\int_{\mathfrak{a}^+}\frac{\partial \nu_t}{\partial \xi_{\infty}}e^{-\nu_t}dx=0.
\]
The difference with what we have done before is that we fix $\xi_{\infty}$ instead 
of considering $\xi_t$. 

Since $\xi_{\infty} \in \mathfrak{a}_{t}$, we have 
$\frac{\partial j}{\partial \xi_{\infty}}=0$ and so we deduce from the vanishing 
of the integral the following equality, valid for $t<t_{\infty}$.
\[
t\int_{\mathfrak{a}^+}\frac{\partial u_t}{\partial \xi_{\infty}} e^{-\nu_t}dx
=(t-1)\int_{\mathfrak{a}^+}\frac{\partial u_{\mathrm{ref}}}{\partial \xi_{\infty}} e^{-\nu_t}dx
\]
The left hand side term is zero because we assumed 
\[
0=\left<\mathrm{bar}_{DH}(2P^+)-4 \rho ,\xi_{\infty}\right>=\left<\mathrm{bar}_{DH}(2P^+) ,\xi_{\infty}\right>.
\]

We thus have, for all $t<t_{\infty}$, 
\[
\int_{\mathfrak{a}^+}\frac{\partial u_{\mathrm{ref}}}{\partial \xi_{\infty}} e^{-\nu_t}dx=0.
\]
This is a contradiction: 
let $m:= \mathrm{min}\{v(\xi)~;~\xi \in \mathfrak{a}, |\xi|=1\}>0$. 
For any $\delta>0$ fixed, there exists an $\epsilon>0$ such that 
if $t_{\infty}-\epsilon<t<t_{\infty}$, 
$\frac{\partial u_{\mathrm{ref}}}{\partial \xi_{\infty}}\geq m/2$ on 
$B(x_t,\delta)$. This is because $|x_t|$ goes to $\infty$ and $u_{\mathrm{ref}}$ 
is asymptotic to $v$. 
Choose now $\delta=\delta(m/4)$ given by Lemma~\ref{Ball}, then for $t$ close 
to $t_{\infty}$, we obtain 
\[
\int_{\mathfrak{a}^+}\frac{\partial u_{\mathrm{ref}}}{\partial \xi_{\infty}} e^{-\nu_t}dx
\geq m/4>0.
\qedhere
\]
\end{proof}

Combined with the obstruction proved earlier, it ends the proof of Theorem~A.

\begin{exa}
Consider the example $X_1$, which we recall is the wonderful compactification of 
$\mathrm{PGL}_3(\mathbb{C})$. The corresponding root system is $A_2$.
We denote by $\alpha_1$ and $\alpha_2$ the simple roots.
The third positive root is then $\alpha_1 + \alpha_2$, and $2\rho=2(\alpha_1+\alpha_2)$.
For $p=x\alpha_1 + y\alpha_2$,
\[
\prod_{\alpha\in \Phi^+} \left<\alpha , p\right>^2 = (x-y/2)^2(-x/2+y)^2(x/2+y/2)^2.
\]

We computed the coordinates $x$ and $y$ of the barycenter $\mathrm{bar}_{DH}(P^+)$
and obtained 
\[
\mathrm{bar}_{DH}(P^+) = \frac{24641}{9888} (\alpha_1+\alpha_2).
\]
As a consequence, $X_1$ admits a Kähler-Einstein metric.
Figure~\ref{BarX1} gives a representation of $P^+$, where the cross is the barycenter, 
and the convex cone delimited by the dashed lines is $2\rho + \Xi$.

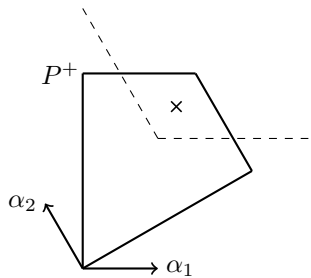
\begin{figure}
\centering
\begin{tikzpicture}[scale=1]
\draw[thick] (0,0) -- (9/4,3*1.73/4);
\draw[thick] (9/4,3*1.73/4) -- (3/2,3*1.73/2);
\draw[thick] (3/2,3*1.73/2) -- (0,3*1.73/2);
\draw[thick] (0,3*1.73/2) -- (0,0);
\draw (-0.3,3*1.73/2) node {$P^+$};
\draw[->, thick] (0,0) -- (1,0);
\draw (1.3,0) node {$\alpha_1$};
\draw[->, thick] (0,0) -- (-1/2,1.73/2);
\draw (-1/2-0.3,1.73/2) node {$\alpha_2$};
\draw[dashed] (1,1.73) -- (3,1.73);
\draw[dashed] (1,1.73) -- (0,2*1.73);
\draw[semithick] (2.49/2-0.07,1.73*2.49/2-0.07) -- (2.49/2+0.07,1.73*2.49/2+0.07);
\draw[semithick] (2.49/2+0.07,1.73*2.49/2-0.07) -- (2.49/2-0.07,1.73*2.49/2+0.07);
\end{tikzpicture}
\caption{Barycenter for $X_1$}
\label{BarX1}
\end{figure}

\end{exa}

\subsection{Lower bound on $R(X)$}
\label{sec_lower}

Assume first that $t_{\infty} <1$, then we can write 
\[
\frac{t_{\infty}}{t_{\infty}-1} \left<\mathrm{bar}_{DH}(2P^+)-4 \rho ,\xi_{\infty}\right> 
= v(\xi_{\infty})-\left<4\rho,\xi_{\infty}\right>,
\]
or 
\[
\left<4\rho + \frac{t_{\infty}}{1-t_{\infty}} (-\mathrm{bar}_{DH}(2P^+) 
				+4 \rho ),\xi_{\infty}\right>
= v(\xi_{\infty}).
\]

For $t=0$, we have $4\rho \in \mathrm{Int}(2P)$. 
The function $t\mapsto t/(1-t)$ is strictly increasing and its image is $[0,\infty[$.
Besides, since $v$ is the support function of $2P$, the value $v(\xi_{\infty})$ is attained 
by $\left< m , \xi_{\infty} \right>$ if and only if $m$ is in the supporting hyperplane of $2P$ 
defined by $\xi_{\infty}$.
We deduce that necessarily $t_{\infty}$ is the unique value of $t$ for which 
\[
4\rho + \frac{t}{1-t} (-\mathrm{bar}_{DH}(2P^+) +4 \rho) \in \partial(-\Xi + 2P^+),
\]
if it exists. If it doesn't exist, then $t_{\infty}=1$. 

Combining this with the upper bound on $R(X)$, we have proved Theorem~C.

\begin{exa}
This allows to compute exactly the greatest Ricci lower bound for $X_2$,
which is 
\[
R(X_2)=\frac{1046175339}{1236719713}.
\]
\end{exa}

\bibliographystyle{alpha}
\bibliography{biblio}

\end{document}